\documentclass[12pt]{article}%
\usepackage{amsfonts}
\usepackage{amsmath}
\usepackage{amssymb}
\usepackage{comment}
\usepackage{graphicx}
\usepackage{geometry}%
\providecommand{\U}[1]{\protect\rule{.1in}{.1in}}
\newtheorem{theorem}{Theorem}

\newtheorem{corollary}[theorem]{Corollary}

\newtheorem{example}[theorem]{Example}

\newtheorem{lemma}[theorem]{Lemma}

\newtheorem{proposition}[theorem]{Proposition}
\newtheorem{remark}[theorem]{Remark}

\newenvironment{proof}[1][Proof]{\noindent\textbf{#1.} }{\ \rule{0.5em}{0.5em}}
\newenvironment{prooff}[1]{\begin{trivlist}\item {\it
\bf Proof}\quad} {\qed\end{trivlist}}
\newcommand{\qed}{\nopagebreak\hspace*{\fill}
{\vrule width6pt height6ptdepth0pt}\par}

\begin{document}

\title{Stochastic transport by Gaussian noise}
%  with regularity greater than $1/2$}

\author{
  {\sc Franco FLANDOLI}
  \thanks{Scuola Normale Superiore,
    Piazza dei Cavalieri, 7 - 56126 Pisa, Italy.
   		E-mail:{ \tt franco.flandoli@sns.it}} 
	\ {\sc and}\ {\sc Francesco RUSSO} 
	\thanks{ENSTA Paris, Institut Polytechnique de Paris.
		Unit\'e de Math\'ematiques Appliqu\'ees (UMA). 
		E-mail:{\tt  francesco.russo@ensta.fr}.
}}

\date{October 17th 2025}

%\author{Franco Flandoli and Francesco Russo}
\maketitle

\begin{abstract}
Diffusion with stochastic transport is investigated here when the random
driving process is a very general Gaussian process, including Fractional
Brownian motion. The purpose is the comparison with a deterministic PDE, which
in certain cases represents the equation for the mean value. From this
equation we observe a reduced dissipation property for small times and an
enhanced diffusion for large times, with respect to delta correlated noise
when regularity is higher than the one of Brownian motion, a fact interpreted
qualitatively here as a signature of the modified dissipation observed for 2D
turbulent fluids due to the inverse cascade. We give results also for the
variance of the solution and for a scaling limit of a two-component noise input.

\end{abstract}


\medskip\noindent  {\bf 2020  AMS-classification}: 76M35; 
60H07; 60H15; 60G22; 76F55;
35K58.

\section{Introduction\label{Section Intro}}

This work investigates the dissipation properties of a stochastic transport
term of Fractional Brownian Motion (FBM henceforth)\ type with Hurst parameter
$H>\frac{1}{2}$, or more generally a Gaussian process with H\"{o}lder paths of
exponents greater than $\gamma_{0}>\frac{1}{2}$, compared to those of Brownian motion.

Starting from the paper \cite{Galeati}, several works proved that a suitable
scaling limit of a Brownian transport term lead to effects similar to those of
an additional dissipation or viscosity, see for instance \cite{Agresti},
\cite{CarigiLuongo}, \cite{DebPapp}, \cite{FGL}, \cite{FGLregular},
\cite{FGLRoyal}, \cite{FGL2}, \cite{FHP}, \cite{FlaHof}, \cite{FlaLuo},
\cite{FlaLuo2023}, \cite{FlaLuoLuongo}, \cite{GalLuo0}, \cite{GalLuo},
\cite{Hofmanova}.\ An open problem is the extension and possibly modification
of these results when Brownian motion is replaced by fractional Brownian
motion (FBM). Proving such scaling limits in the case of (FBM) looks very
difficult, since it requires to handle a non-commutative framework which
provokes considerable difficulties, as discussed in Section
\ref{Section non commut}. In this paper we prove a preliminary result
indicating that a new diffusion property could arise. We limit ourselves, in
our main quantitative result, to a commutative case and prove properties
related to average and variance of solution. With little abuse of language, by commutative case, we mean the situation when the coefficient $\sigma_k(x)$
of equation \eqref{eq 1} below are constant, namely they do not depend on $x$,
see also the comments at the end of the introduction.

We investigate the following model for the diffusion of a passive scalar
$\theta_{\epsilon}\left(  t,x\right)  $ (e.g. the temperature of the fluid),
with $t\geq0$ and $x\in\mathbb{R}^{2}$:\
\begin{align}
\partial_{t}\theta_{\epsilon}\left(  t,x\right)   &  =\kappa\Delta
\theta_{\epsilon}\left(  t,x\right)  +\sum_{k\in K}\left(  \sigma_{k}\left(
x\right)  \cdot\nabla\right)  \theta_{\epsilon}\left(  t,x\right)
\frac{d\mathcal{G}_{t}^{k,\epsilon}}{dt}\label{eq 1}\\
\theta_{\epsilon}|_{t=0} &  =\theta_{0},\nonumber
\end{align}
where $\kappa>0$ is a (small) diffusion constant (in most part of this work
also $\kappa=0$ is admitted), $K$ is a finite index set, $\sigma_{k}\left(
x\right)  $ are smooth divergence free vector fields and $\mathcal{G}%
_{t}^{k,\epsilon}$ are $\epsilon$-regularization of stationary increment
Gaussian processes $G_{t}^{k}$, $\epsilon>0$:
\begin{equation}
\mathcal{G}_{t}^{k,\epsilon}=\int_{0}^{t}\frac{G_{s+\epsilon}^{k}%
-G_{(s-\epsilon)_{+}}^{k}}{2\epsilon}ds,\label{def gk}%
\end{equation}
where, here and below, we denote by $(s-\epsilon)_{+}$ the maximum between
$s-\epsilon$ and zero. Hence the driving random field $\frac{d\mathcal{G}%
_{t}^{k,\epsilon}}{dt}$ is stationary. The choice to work in $\mathbb{R}^{2}$
is only due to the motivation of the inverse cascade, but the results proved
here hold in any space dimension. The interpretation is that the velocity
field $u_{\epsilon}\left(  t,x\right)  $ of the fluid is modeled by the
stationary random field%
\[
u_{\epsilon}\left(  t,x\right)  =\sum_{k\in K}\sigma_{k}\left(  x\right)
\frac{d\mathcal{G}_{t}^{k,\epsilon}}{dt}.
\]
Some of the works quoted above had in mind the case when the fluid structures
$\sigma_{k}\left(  x\right)  $ were
%%%% FRANCO. c'era where
small, modeling small-space-scale
turbulence; accordingly, it was natural to idealize the time-structure
assuming very small time-correlation,
%%% FRANCO $H$ piccolo?
hence $\frac{d\mathcal{G}_{t}^{k,\epsilon}}{dt}$
related to white noise. Here we have in mind larger space
structures $\sigma_{k}\left(  x\right)  $ and longer time correlation of
$\frac{d\mathcal{G}_{t}^{k,\epsilon}}{dt}$, as it is realized when
$\mathcal{G}_{t}^{k,\epsilon}$ is related to FBM with $H>1/2$.

We follow the philosophy that the physical model is the family of equations
parametrized by $\epsilon>0$. Of course it is mathematically interesting to
investigate the limit as $\epsilon\rightarrow0$ in itself but this is not the
purpose of this work. On the contrary, we compute observations, especially
mean values and take the limit as $\epsilon\rightarrow0$ of their results,
getting clean final expressions for the observed quantities.

Our aim is understanding the dissipation, in an average sense, for $H>1/2$,
compared to the white noise case. Assume that the Gaussian fields $G_{t}^{k}$
are independent and equally distributed. We thus introduce the \textit{mean
field equation} associated to (\ref{eq 1})
\begin{align}
\partial_{t}\overline{\theta}\left(  t,x\right)   &  =\kappa\Delta
\overline{\theta}\left(  t,x\right)  +\frac{d\gamma\left(  t\right)  }%
{dt}\left(  \mathcal{L}\overline{\theta}\left(  t\right)  \right)  \left(
x\right)  \label{eq 2}\\
\overline{\theta}|_{t=0} &  =\theta_{0}.\nonumber
\end{align}
Here $\mathcal{L}$ is the elliptic operator (possibly non-uniformly elliptic)%
\begin{align}
\left(  \mathcal{L}f\right)  \left(  x\right)   &  =\operatorname{div}\left(
Q\left(  x,x\right)  \nabla f\left(  x\right)  \right)  \label{diff oper}\\
Q\left(  x,y\right)   &  =\sum_{k\in K}\sigma_{k}\left(  x\right)
\otimes\sigma_{k}\left(  y\right)  \nonumber
\end{align}
and $\gamma\left(  t\right)  $ is the variance function of the $G_{t}^{k}$:
\[
\gamma\left(  t\right)  =Var\left(  G_{t}^{k}\right)  .
\]
When we deduce equation (\ref{eq 2}) we restrict ourselves to $\frac
{d\gamma\left(  t\right)  }{dt}\geq0$ but this investigation opens the door to
the possibility of negative viscosities, mentioned for instance by
\cite{Wirth}. In relevant cases
\[
\left(  \mathcal{L}\overline{\theta}\left(  t\right)  \right)  \left(
x\right)  =\kappa_{T}\Delta\overline{\theta}\left(  t,x\right)  ,
\]
where $\kappa_{T}>0$ may be called eddy dissipation constant, and
\[
\frac{d\gamma\left(  t\right)  }{dt}\sim t^{2H-1},%
\]
for small $t$. In this case, and choosing the case $\kappa=0$ (admissible in
our results) for an easier interpretation of the results, the mean field
equation takes the form%
\[
\partial_{t}\overline{\theta}\left(  t,x\right)  =t^{2H-1}\kappa_{T}%
\Delta\overline{\theta}\left(  t,x\right)  .
\]
Thus we see that, compared to the Brownian case ($H=1/2$) where $t^{2H-1}%
\kappa_{T}$ is constant, when $H>1/2$ the diffusion coefficient $t^{2H-1}%
\kappa_{T}$ is small for small times, and large for large times: dissipation
is depleted for small times, enhanced for large ones. In Section
\ref{subsect turbulence} we propose an interpretation of these facts in
turbulence theory.

In addition, we get an equation for the limit as $\epsilon\rightarrow0$ of the
variance, which is%

\[
\partial_{t}V\left(  t,x\right)  =2\kappa\Delta V\left(  t,x\right)
+\frac{d\gamma\left(  t\right)  }{dt}\left(  \left(  \mathcal{L}V\left(
t\right)  \right)  \left(  x\right)  +2\sum_{k\in K}\left(  \left(  \sigma
_{k}\cdot\nabla\right)  \overline{\theta}\left(  t,x\right)  \right)
^{2}\right)  .
\]
Taking $\kappa=0$ again to make a simple example, noticing that $V=0$ at time
$t=0$, we get for small $t$%
\begin{align*}
\overline{\theta}\left(  t\right)   &  \sim\left(  \mathcal{L}\theta
_{0}\right)  t^{2H}\\
\sqrt{V\left(  t\right)  } &  \sim\sqrt{2\sum_{k\in K}\left(  \left(
\sigma_{k}\cdot\nabla\right)  \theta_{0}\right)  ^{2}}t^{H}.
\end{align*}
When $H>1/2$ and time is small, not only the average dissipation is
infinitesimal with respect to the Brownian case, but also a confidence
interval around the average is small.

Among the interests in the computation of the commutative case there is its
generality in terms of noise, which in particular covers all $H\in\left(
0,1\right)  $ and even beyond. We refer in particular to a stochastic analysis
related to a very singular covariance process, see e.g. \cite{mocioalca}.

There are two (potential) links between the stochastic equation (\ref{eq 1})
and its mean field equation (\ref{eq 2}): the first one is the one proved
here, namely (\ref{eq 2}) is the equation for the limit of the average. The
potential second one would be that equation (\ref{eq 2}) is the limit of
single realizations of $\theta_{\epsilon}\left(  t,x\right)  $; namely that
realizations of $\theta_{\epsilon}\left(  t,x\right)  $ concentrate around the
average. This second class of results is precisely the one started from the
seminal work \cite{Galeati}, proved in the case of Brownian motion. In the
general Gaussian case treated here this second link is much less obvious. In
the last Section \ref{Section non commut} we move a few steps in the direction
of the non-commutative case. This case is very difficult and our understanding
is only fragmentary. Subsection \ref{Section Galeati} describes an idealized
model of turbulent 2D flow incorporating the idea of inverse cascade in the
simplest possible way: the noise is divided into two components, one
small-space-scale and white noise in time, the other larger-space-scale and
correlated in time. In the larger one the space structures are constant in
space. We prove that the smaller scales produce the effect predicted by the
Boussinesq hypothesis, while the larger ones remain in their form. It is also
an example of reduction to the commutative case. Subsection
\ref{Section Galeati} is finally devoted to show the difficulty arising in the
non-commutative case, where a commutator arises as a remainder in the link
between the true expected value of the solution and the mean field equation.

Let us finally remark that the model presented here could be of interest also
in connection with other research directions on stochastic transport, not
necessarily related to Boussinesq assumption and inverse cascade. In
particular, since we have modeled larger space structures, the connection with
the general activity reported in \cite{Chapron}, originated by the seminal
work \cite{Holm}, see also \cite{Chapron2}, \cite{Crisan0}, \cite{Crisan},
\cite{Epiharati}, \cite{Harouna}, \cite{HolmLues}, \cite{Memin}.

As a last comment, let us recall that the case of commuting noise can also be
approached by means of semigroups associated to the different noise terms%
\[
e^{B_{k}\mathcal{G}_{t}^{k,\epsilon}},%
\]
where $B_{k}$ is the linear operator $\sigma_{k}\left(  x\right)  \cdot\nabla$
with suitable domain. In the commuting case,
we have $B_k B_h f = B_h B_k f$, for every $h, k \in K$ and
every smooth function $f$, so that
these semigroups commute. Let us
mention that the approach to stochastic transport by this semigroup method was
introduced long time ago by Giuseppe Da Prato (to whom this paper is
dedicated), Mimmo Iannelli and Luciano Tubaro, see for instance \cite{DIT}. 

\subsection{Tentative interpretation for 2D
turbulence\label{subsect turbulence}}

The model above of stochastic transport is certainly too idealized to deduce
consequences for the very complex behavior of turbulent 2-dimensional fluids.
Nevertheless, let us dream of a few potential links between the results of
this paper and 2D turbulence, keeping in mind that all the statements made
here require more careful investigation when the mathematical technology will
allow it. 

Concerning the idealization, a main one is that 2D turbulence refers to the
properties of the solution of 2D Euler or Navier-Stokes equations, not just
passive scalar transport equations. The reason why we take the risk of such a
link is the fact that the diffusive limit result for passive scalars of
\cite{Galeati} has been extended to the 2D Euler or Navier-Stokes equations
\cite{FGL}, \cite{FGL2}. Hence maybe in the future, if the diffusive limit may
be implemented for fractional Gaussian noise, the results could move to
nonlinear equations (opposite to the computation of mean values which does not close).

Thus, with this remark in mind, the central idea we want to express is that in
2D turbulence small space and time scale turbulent structures (eddies or
vortices) have a tendency to gather in larger structures, called inverse
cascade. The larger structures have a larger space-scale and also a larger
time-scale, a longer time correlation; the velocity has a tendency to remain
in the same direction for a longer time, compared to the small vortex
structures. Fractional Gaussian noise in the persistent regime $H>1/2$ could
be a choice to model such behavior. Also from the viewpoint of space the
structures should be larger; here, in the commutative case, we assume they are
constant in space, which is certainly an assumption made for mathematical
convenience, but it is also an idealization of large scale structures. Thus,
summarizing, our model of constant vector fields with fractional Gaussian
noise with $H>1/2$ could be seen as a strong idealization of the larger
structures which appear spontaneously (and inevitably) in 2D inverse cascade. 

If so, the conclusion of our computation is that for small amounts of time the
diffusive properties of such a larger scale noise is weaker than those of
classical white noise. But for longer times the diffusion is stronger, namely
the information reaches longer distances. Both facts are coherent with the
intuition of what larger structures do in a fluid: they cannot mix and diffuse
instantaneously as smaller structures do, but they transport information far
away faster, thanks to their better coherence. 

Although FBM\ is the paradigmatic example, it may be useful to have more
degrees of freedom in the statistical choice (see examples in \cite{Apollin1},
\cite{Apolinario}, \cite{Chevillard}) and thus Gaussian process with
H\"{o}lder paths of exponents greater than $\gamma_{0}>\frac{1}{2}$ may be a
convenient framework. 

\section{Preparatory material\label{Section preparatory}}

The material of this section is adapted from Chapter 1. of \cite{nualart}, see
\cite{krukthesis}, in particular Section 10, for a more explicit formulation
and summarized here for the reader's convenience. See also
\cite{ohashi2022rough} for some more recent developments.

We consider a Gaussian process $G:=\left(  G^{k};k\in K\right)  $ in
$\mathbb{R}^{N}$, $N=Card\left(  K\right)  $, whose components are independent
and identically distributed. Denote by $\mathcal{H}$ the self-reproducing
kernel space of $G^{1}$, with scalar product $\left\langle \cdot
,\cdot\right\rangle _{\mathcal{H}}$. Recall that $\left\langle 1_{\left[
0,t\right]  },1_{\left[  0,s\right]  }\right\rangle _{\mathcal{H}}$ gives us
the covariance function of $G^{1}$.

\begin{example}
The FBM with $H>1/2$ has covariance function given by%
\[
R\left(  t,s\right)  =\frac{1}{2}\left(  s^{2H}+t^{2H}-\left\vert
t-s\right\vert ^{2H}\right)  =\alpha_{H}\int_{0}^{t}\int_{0}^{s}\left\vert
r-u\right\vert ^{2H-2}dudr,
\]
for a suitable constant $\alpha_{H}>0$ and therefore%
\[
\left\langle f,g\right\rangle _{\mathcal{H}}=\alpha_{H}\int_{0}^{T}\int%
_{0}^{T}\left\vert r-u\right\vert ^{2H-2}f\left(  r\right)  g\left(  u\right)
dudr.
\]

\end{example}

\begin{example}
A model which seems to fit better our intuition of the intermediate vortex
structures of a 2D turbulent fluid is%
\begin{align*}
R\left(  t,s\right)   &  =\alpha_{H,\lambda}\int_{0}^{t}\int_{0}^{s}\left\vert
r-u\right\vert ^{2H-2}e^{-\lambda\left\vert r-u\right\vert }dudr\\
\left\langle f,g\right\rangle _{\mathcal{H}}  &  =\alpha_{H,\lambda}\int%
_{0}^{T}\int_{0}^{T}\left\vert r-u\right\vert ^{2H-2}e^{-\lambda\left\vert
r-u\right\vert }f\left(  r\right)  g\left(  u\right)  dudr,
\end{align*}
with $\lambda>0$, for a suitable constant $\alpha_{H,\lambda}>0$. Indeed,
\begin{align*}
\mathbb{E}\left[  \overset{\cdot}{\mathcal{G}}_{t}^{k,\epsilon}\overset{\cdot
}{\mathcal{G}}_{s}^{k,\epsilon}\right]   &  =\left(  2\epsilon\right)
^{-2}\mathbb{E}\left[  \left(  G_{t+\epsilon}-G_{t-\epsilon}\right)  \left(
G_{s+\epsilon}-G_{s-\epsilon}\right)  \right] \\
&  =\left(  2\epsilon\right)  ^{-2}\left(  R\left(  t+\epsilon,s+\epsilon
\right)  -R\left(  t+\epsilon,s-\epsilon\right)  -R\left(  t-\epsilon
,s+\epsilon\right)  +R\left(  t-\epsilon,s-\epsilon\right)  \right) \\
&  \rightarrow\partial_{t}\partial_{s}R\left(  t,s\right)  =\left\vert
t-s\right\vert ^{2H-2}e^{-\lambda\left\vert t-s\right\vert }.
\end{align*}
This process develops, locally in time, the same correlation structure of FBM,
but loses memory in the long time, closely to the fact that also large scale
vortex structures are like a birth and death process, they do not persist to infinity.
\end{example}
%{\bf FINE VERIFICHE FRANCESCO}
\begin{remark}
\label{remark large times}In the case of the previous example, call $\tau>0$ a
time of decorrelation of the vortex structures we want to model and take
$\lambda=1/\tau$. We have%
\[
\frac{d\gamma\left(  t\right)  }{dt}=2\alpha_{H,\lambda}\int_{0}^{t}%
r^{2H-2}e^{-\lambda r}dr,
\]
which behaves like
\[
\frac{d\gamma\left(  t\right)  }{dt}\sim t^{2H-1},%
\]
for small $t$ but not for large ones. The limit as $t\rightarrow\infty$ of
$\frac{d\gamma\left(  t\right)  }{dt}$ is given by ($\Gamma\left(  r\right)  $
denotes the Gamma function)
\[
2\alpha_{H,\lambda}\int_{0}^{\infty}r^{2H-2}e^{-\lambda r}dr=2\alpha
_{H,\lambda}\Gamma\left(  2H-1\right)  \tau^{2H-1}\sim\tau^{2H-1},
\]
so it remains small if $\tau$ is small.
\end{remark}

Let $\Phi:\mathbb{R}^{N}\rightarrow\mathbb{R}$ be smooth an bounded and let
$\varphi_{k}\in\mathcal{H}$, $k\in K$. Set
\[
Y=\Phi\left(  \int_{0}^{T}\varphi_{k}dG^{k};k\in K\right),
\]
where $\int_{0}^{T}\varphi_{k}dG^{k}$ are Wiener integrals. Then the Malliavin
derivative%
\[
DY=\left(  D^{\left(  k\right)  }Y;k\in K\right)
\]
is a vector process given by
\[
D_{r}^{\left(  \ell\right)  }Y=\left(  \partial_{\ell}\Phi\right)  \left(
\int_{0}^{T}\varphi_{k}dG^{k};k\in K\right)  \varphi_{\ell}\left(  r\right),
\]
see e.g. (8.3) \cite{krukthesis} and Section 6 of \cite{krukJFA}.
The integration by parts on Wiener space states that, if $Z=\left(  Z^{k};k\in
K\right)  $ is a Malliavin smooth vector of processes, then%
\[
\mathbb{E}\left[  \left\langle DY,Z\right\rangle _{\mathcal{H}} \right]
=\mathbb{E}\left[  Y \delta\left(  Z\right)  \right] ,
\]
where $\delta$ is the Skorohod integral (divergence operator), which has zero
expectation, among other properties. For instance, if $G$ is an $N$%
-dimensional Brownian motion, we get%
\[
\delta\left(  Z\right)  =\sum_{k\in K}\int_{0}^{T}Z^{k}dG^{k}.
\]
If $\varphi\in\mathcal{H}$ then the Wiener integral $\int_{0}^{T} g dG$
coincides with the Skorohod integral $\delta(g)$. Key to our developments
below is rewriting the terms
\[
\left(  \sigma_{k}\left(  x\right)  \cdot\nabla\right)  \theta_{\epsilon
}\left(  t,x\right)  \frac{G_{s+\epsilon}^{k}-G_{s-\epsilon}^{k}}{2\epsilon}%
\]
as a suitable mean zero term plus a term which remains when taking the mean
and the limit as $\epsilon\rightarrow0$ and possibly is "closed", namely
expressed in terms of the mean of the limit of $\theta_{\epsilon}\left(
t,x\right)  $. The first step is identifying a suitable mean zero part and it
will be the Skorohod integral. The second step is understanding the remaining
term, which is the so-called trace.

If $X$ is a Malliavin smooth stochastic process, we have
\begin{align*}
X_{s}\frac{G_{s+\epsilon}^{k}-G_{(s-\epsilon)_{+}}^{k}}{2\epsilon}  &
=X_{s}\frac{1}{2\epsilon}\int_{(s-\epsilon)_{+}}^{s+\epsilon}\delta G_{s}%
^{k}\\
&  =\frac{1}{2\epsilon}\int_{(s-\epsilon)_{+}}^{s+\epsilon}X_{s}\delta
G_{s}^{k}+\left\langle D_{\cdot}^{\left(  k\right)  }X_{s},\frac{1}{2\epsilon
}1_{\left[  (s-\epsilon)_{+},s+\epsilon\right]  }\right\rangle _{\mathcal{H}},
\end{align*}
hence we get the formula%

\begin{equation}
\int_{0}^{t}X_{s}\frac{G_{s+\epsilon}^{k}-G_{(s-\epsilon)_{+}}^{k}}{2\epsilon
}ds=M_{t}+\int_{0}^{t}\left\langle D_{\cdot}^{\left(  k\right)  }X_{s}%
,\frac{1}{2\epsilon}1_{\left[  (s-\epsilon)_{+},s+\epsilon\right]
}\right\rangle _{\mathcal{H}}ds, \label{basic formula}%
\end{equation}
where $M$ is the mean zero process%
\[
M_{t}=\int_{0}^{t}\left(  \frac{1}{2\epsilon}\int_{(s-\epsilon)_{+}%
}^{s+\epsilon}X_{s}\delta G_{s}^{k}\right)  ds.
\]
Formula (\ref{basic formula}) is used several times below.

\section{The commutative case\label{Section Commutative}}

We consider the equation%
\begin{equation}
\theta_{\epsilon}\left(  t,x\right)  =\theta_{0}\left(  x\right)  +\int%
_{0}^{t}\kappa\Delta\theta_{\epsilon}\left(  s,x\right)  ds+\sum_{k\in K}%
\int_{0}^{t}\left(  \sigma_{k}\cdot\nabla\right)  \theta_{\epsilon}\left(
s,x\right)  \frac{G_{s+\epsilon}^{k}-G_{(s-\epsilon)_{+}}^{k}}{2\epsilon}ds,
\label{F 1}%
\end{equation}
when the transport noise vector fields are constant, $\sigma_{k}\in
\mathbb{R}^{2}$. We assume that $G^{k}$, $k\in K$, are independent mean zero
Gaussian processes, starting at zero, equally distributed.

One can solve equation (\ref{F 1}) by Fourier transform. We use the convention
that
\begin{align*}
\widehat{f}\left(  \xi\right)   &  =\int_{\mathbb{R}^{d}}e^{-2\pi i\xi\cdot
x}f\left(  x\right)  dx\\
f\left(  x\right)   &  =\int_{\mathbb{R}^{d}}e^{2\pi i\xi\cdot x}%
\widehat{f}\left(  \xi\right)  d\xi,
\end{align*}
when the notations are meaningful, in a classical or generalized sense. The
equation in Fourier transform reads
\begin{equation}
\widehat{\theta_{\epsilon}}\left(  t,\xi\right)  =\widehat{\theta_{0}}\left(
\xi\right)  -\kappa\left\vert \xi\right\vert ^{2}\int_{0}^{t}\widehat{\theta
_{\epsilon}}\left(  s,\xi\right)  ds+i\sum_{k\in K}\left(  \sigma_{k}\cdot
\xi\right)  \int_{0}^{t}\widehat{\theta_{\epsilon}}\left(  s,\xi\right)
\frac{G_{s+\epsilon}^{k}-G_{(s-\epsilon)_{+}}^{k}}{2\epsilon}ds \label{F 2}%
\end{equation}
and, being decoupled with respect to the frequency variable $\xi$, it is
pointwise meaningful as a complex-valued ordinary differential equation
parametrized by $\xi$. In fact it can be explicitly solved:%
\begin{equation}
\widehat{\theta_{\epsilon}}\left(  t,\xi\right)  =\widehat{\theta_{0}}\left(
\xi\right)  \exp\left(  -\kappa\left\vert \xi\right\vert ^{2}t+i\sum_{k\in
K}\left(  \sigma_{k}\cdot\xi\right)  \mathcal{G}_{t}^{k,\epsilon}\right) ,
\label{F explicit}%
\end{equation}
where $\mathcal{G}_{t}^{k,\epsilon}$ is defined by (\ref{def gk}). We
introduce the following assumption.

\textbf{Assumption A}

i) $G$ is a Gaussian continuous process with stationary increments, vanishing
at zero.
%%%%

ii) $\gamma\left(  t\right)  =Var\left(  G_{t}^{1}\right)  $ is a bounded
variation function.

\begin{remark}
The property of stationary increments can be relaxed, but we keep it as it is
to avoid complications.
\end{remark}

\begin{remark}
Of course this includes the case of $G$ being a Fractional Brownian motion
with any Hurst index $H$.
\end{remark}

\begin{lemma}
\label{lemma1}Suppose Assumption A. We denote
\[
\overset{\cdot}{\mathcal{V}}_{\epsilon}\left(  t\right)  =\frac{1}{\left(
2\epsilon\right)  ^{2}}\int_{0}^{t}\left\langle 1_{\left[  (t-\epsilon
)_{+},t+\epsilon\right]  },1_{\left[  (s-\epsilon)_{+},s+\epsilon\right]
}\right\rangle _{\mathcal{H}}ds
\]
and%
\begin{equation}
\mathcal{V}_{\epsilon}\left(  \tau\right)  =\int_{0}^{\tau}\overset{\cdot
}{\mathcal{V}}_{\epsilon}\left(  t\right)  dt. \label{F 3}%
\end{equation}
Then the measure $d\mathcal{V}_{\epsilon}\left(  t\right)  $ converges weak
star to $d\gamma\left(  t\right)  $, namely%
\[
\int_{0}^{T}\varphi\left(  t\right)  d\mathcal{V}_{\epsilon}\left(  t\right)
\rightarrow\int_{0}^{T}\varphi\left(  t\right)  d\gamma\left(  t\right)  ,
\]
for every $\varphi\in C\left(  \left[  0,T\right]  \right)  $.
\end{lemma}

\begin{remark}
The integral in (\ref{F 3}) is a well-defined Bochner integral in
$\mathcal{H}$.
\end{remark}

\begin{remark}
One has $D_{r}^{\left(  k\right)  }\mathcal{G}_{t}^{k^{\prime},\epsilon}=0$
for $k^{\prime}\neq k$ and%
\begin{equation}
D_{r}^{\left(  k\right)  }\mathcal{G}_{t}^{k,\epsilon}=\frac{1}{2\epsilon}%
\int_{0}^{t}1_{\left[  (s-\epsilon)_{+},s+\epsilon\right]  }\left(  r\right)
ds. \label{approximate derivative}%
\end{equation}
Indeed (the case $k^{\prime}\neq k$ is similar),
\[
D_{r}^{\left(  k\right)  }\int_{(s-\epsilon)_{+}}^{s+\epsilon}\delta G_{u}%
^{k}=1_{\left[  (s-\epsilon)_{+},s+\epsilon\right]  }\left(  r\right) .
\]

\end{remark}

\begin{prooff} \ (of Lemma \ref{lemma1}).
By \eqref{F 3}, we have
  %We have, for $\tau\geq\epsilon$,
\begin{align*}
\mathcal{V}_{\epsilon}\left(  \tau\right)   &  =\int_{0}^{\tau}\frac
{1}{\left(  2\epsilon\right)  ^{2}}\int_{0}^{t}\left\langle 1_{\left[
(t-\epsilon)_{+},t+\epsilon\right]  },1_{\left[  (s-\epsilon)_{+},s+\epsilon\right]  }\right\rangle _{\mathcal{H}}dsdt.
%   \\
% &  =:\mathcal{V}_{\epsilon}\left(  \epsilon\right)  +\widetilde{\mathcal{V}%
% }_{\epsilon}\left(  \tau\right) ,
\end{align*}
Therefore, for $\tau \ge \epsilon$, we can decompose
$$  \mathcal{V}_{\epsilon}\left(  \tau\right) =
\mathcal{V}_{\epsilon}\left(  \epsilon\right)  +\widetilde{\mathcal{V} }_{\epsilon}\left(  \tau\right),$$
where $\widetilde{\mathcal{V}}_{\epsilon}\left(  \tau\right)  =0$ for
$\tau\leq\epsilon$,
\[
\widetilde{\mathcal{V}}_{\epsilon}\left(  \tau\right)  =\int_{\epsilon}^{\tau
}\frac{1}{\left(  2\epsilon\right)  ^{2}}\int_{\epsilon}^{t}\left\langle
1_{\left[  (t-\epsilon)_{+},t+\epsilon\right]  },1_{\left[  (s-\epsilon
)_{+},s+\epsilon\right]  }\right\rangle _{\mathcal{H}}dsdt,
\]
for $\tau\geq\epsilon$. It is not difficult to show that $\mathcal{V}%
_{\epsilon}\left(  \epsilon\right)  \rightarrow0$ as $\epsilon\rightarrow0$.
It remains to show that $d\widetilde{\mathcal{V}}_{\epsilon}\left(  t\right)
$ converges weak star to $d\gamma\left(  t\right)  $. For this it is enough to
show that
\begin{equation}
\widetilde{\mathcal{V}}_{\epsilon}\left(  \tau\right)  \rightarrow\frac{1}%
{2}\gamma\left(  \tau\right)  \text{ for every }\tau\in\left[  0,T\right]
\label{prima da dim}%
\end{equation}
and that $\sup_{\epsilon\in\left(  0,1\right)  }\int_{0}^{T}\left\vert
\overset{\cdot}{\widetilde{\mathcal{V}}_{\epsilon}}\left(  t\right)
\right\vert dt<\infty$, hence that
\begin{equation}
\sup_{\epsilon\in\left(  0,1\right)  }\int_{\epsilon}^{T}\frac{1}{\left(
2\epsilon\right)  ^{2}}\left\vert \int_{\epsilon}^{t}\left\langle 1_{\left[
(t-\epsilon)_{+},t+\epsilon\right]  },1_{\left[  (s-\epsilon)_{+}%
,s+\epsilon\right]  }\right\rangle _{\mathcal{H}}ds\right\vert dt<\infty.
\label{seconda da dim}%
\end{equation}
At this point, for $t\geq s\geq\epsilon$ (denote any one of the $G^{k}$ by
$G$), using stationarity of the increments,%
\begin{align*}
&  \left\langle 1_{\left[  (t-\epsilon)_{+},t+\epsilon\right]  },1_{\left[
(s-\epsilon)_{+},s+\epsilon\right]  }\right\rangle _{\mathcal{H}}\\
&  =Cov\left(  G_{t+\epsilon}-G_{(t-\epsilon)_{+}},G_{s+\epsilon
}-G_{(s-\epsilon)_{+}}\right) \\
&  =Cov\left(  G_{t-s+2\epsilon}-G_{t-s},G_{2\epsilon}\right) \\
&  =Cov\left(  G_{t-s+2\epsilon},G_{2\epsilon}\right)  -Cov\left(
G_{t-s},G_{2\epsilon}\right) \\
&  =\frac{1}{2}\left(  \gamma\left(  t-s+2\epsilon\right)  +\gamma\left(
2\epsilon\right)  -\gamma\left(  t-s\right)  \right) \\
&  =\frac{1}{2}\left(  \gamma\left(  2\epsilon\right)  +\gamma\left(
t-s\right)  -\gamma\left(  t-s-2\epsilon\right)  \right) ,
\end{align*}
with the convention that $\gamma$ is extended by parity for negative
arguments. So%
\begin{align*}
&  \left\langle 1_{\left[  (t-\epsilon)_{+},t+\epsilon\right]  },1_{\left[
(s-\epsilon)_{+},s+\epsilon\right]  }\right\rangle _{\mathcal{H}}\\
&  =\frac{1}{2}\left(  \gamma\left(  t-s+2\epsilon\right)  -\gamma\left(
t-s\right)  \right) \\
&  -\frac{1}{2}\left(  \gamma\left(  t-s\right)  -\gamma\left(  t-s-2\epsilon
\right)  \right) .
\end{align*}
Now, for $\tau\geq\epsilon$, by telescopy,
\begin{align*}
\widetilde{\mathcal{V}}_{\epsilon}\left(  \tau\right)   &  =\int_{\epsilon
}^{\tau}\frac{1}{8\epsilon^{2}}\int_{0}^{t}\left(  \gamma\left(
t-s+2\epsilon\right)  -\gamma\left(  t-s\right)  \right)  dsdt\\
&  -\int_{\epsilon}^{\tau}\frac{1}{8\epsilon^{2}}\int_{0}^{t}\left(
\gamma\left(  t-s\right)  -\gamma\left(  t-s-2\epsilon\right)  \right)  dsdt
\end{align*}%
\begin{align*}
&  =\int_{\epsilon}^{\tau}\frac{1}{8\epsilon^{2}}\int_{0}^{t-\epsilon}\left(
\gamma\left(  s+2\epsilon\right)  -\gamma\left(  s\right)  \right)  dsdt\\
&  -\int_{\epsilon}^{\tau}\frac{1}{8\epsilon^{2}}\int_{0}^{t-\epsilon}\left(
\gamma\left(  s\right)  -\gamma\left(  s-2\epsilon\right)  \right)  dsdt
\end{align*}%
\begin{align*}
&  =\int_{\epsilon}^{\tau}\frac{1}{8\epsilon^{2}}\int_{\epsilon}^{t}\left(
\gamma\left(  s+\epsilon\right)  -\gamma\left(  s-\epsilon\right)  \right)
dsdt\\
&  -\int_{\epsilon}^{\tau}\frac{1}{8\epsilon^{2}}\int_{-\epsilon}^{t-\epsilon
}\left(  \gamma\left(  s+\epsilon\right)  -\gamma\left(  s-\epsilon\right)
\right)  dsdt.
\end{align*}
So, again by telescopy, for $\tau\geq\epsilon$,%
\[
\widetilde{\mathcal{V}}_{\epsilon}\left(  \tau\right)  =\frac{1}{2}\left(
I_{1}\left(  \tau,\epsilon\right)  -I_{2}\left(  \tau,\epsilon\right)  \right)
,
\]
where%
\begin{align*}
I_{1}\left(  \tau,\epsilon\right)   &  =\int_{\epsilon}^{\tau}\frac
{1}{2\epsilon}\int_{t-\epsilon}^{t+\epsilon}\frac{\gamma\left(  s+\epsilon
\right)  -\gamma\left(  s-\epsilon\right)  }{2\epsilon}dsdt\\
I_{2}\left(  \tau,\epsilon\right)   &  =\int_{\epsilon}^{\tau}\frac
{1}{2\epsilon}\int_{-\epsilon}^{0}\frac{\gamma\left(  s+\epsilon\right)
-\gamma\left(  s-\epsilon\right)  }{2\epsilon}dsdt.
\end{align*}
By Fubini theorem, using that
\[
\gamma\left(  s+\epsilon\right)  -\gamma\left(  s-\epsilon\right)
=\int_{s-\epsilon}^{s+\epsilon}d\gamma\left(  r\right) ,
\]
we can easily show that%
\begin{align*}
I_{1}\left(  \tau,\epsilon\right)   &  \rightarrow\gamma\left(  \tau\right) \\
I_{2}\left(  \tau,\epsilon\right)   &  \rightarrow0.
\end{align*}
This shows (\ref{prima da dim}). Concerning (\ref{seconda da dim}), we proceed
similarly. By the same arguments as before we show that%
\begin{align*}
&  \int_{\epsilon}^{T}\frac{1}{\left(  2\epsilon\right)  ^{2}}\left\vert
\int_{\epsilon}^{t}\left\langle 1_{\left[  (t-\epsilon)_{+},t+\epsilon\right]
},1_{\left[  (s-\epsilon)_{+},s+\epsilon\right]  }\right\rangle _{\mathcal{H}%
}ds\right\vert dt\\
&  \leq\int_{\epsilon}^{\tau}\frac{1}{2\epsilon}\left\vert \int_{t-\epsilon
}^{t+\epsilon}\frac{\gamma\left(  s+\epsilon\right)  -\gamma\left(
s-\epsilon\right)  }{2\epsilon}ds\right\vert dt\\
&  +\int_{\epsilon}^{\tau}\frac{1}{2\epsilon}\left\vert \int_{-\epsilon}%
^{0}\frac{\gamma\left(  s+\epsilon\right)  -\gamma\left(  s-\epsilon\right)
}{2\epsilon}ds\right\vert dt.
\end{align*}
We proceed as above, using that%
\[
\left\vert \gamma\left(  s+\epsilon\right)  -\gamma\left(  s-\epsilon\right)
\right\vert \leq\int_{s-\epsilon}^{s+\epsilon}d\left\Vert \gamma\right\Vert
\left(  r\right) ,
\]
where $\left\Vert \gamma\right\Vert $ is the total variation function.
\end{prooff}

\begin{lemma}
\label{lemma link Malliavin derivative to function}$D_{r}^{\left(  k\right)
}\widehat{\theta}_{\epsilon}\left(  t,\xi\right)  $ exists and it is given by%
\begin{equation}
D_{r}^{\left(  k\right)  }\widehat{\theta_{\epsilon}}\left(  t,\xi\right)
=\widehat{\theta_{\epsilon}}\left(  t,\xi\right)  i\left(  \sigma_{k}\cdot
\xi\right)  \frac{1}{2\epsilon}\int_{0}^{t}1_{\left[  (s-\epsilon
)_{+},s+\epsilon\right]  }\left(  r\right)  ds. \label{F Malliavin}%
\end{equation}

\end{lemma}

\begin{proof}
We could deduce this formula from equation (\ref{F 2}) and its uniqueness
property. For shortness, let us use the explicit formula (\ref{F explicit}).
It gives us%
\begin{align*}
D_{r}^{\left(  k\right)  }\widehat{\theta_{\epsilon}}\left(  t,\xi\right)   &
=\widehat{\theta_{\epsilon}}\left(  t,\xi\right)  i\sum_{k^{\prime}\in
K}\left(  \sigma_{k^{\prime}}\cdot\xi\right)  D_{r}^{\left(  k\right)
}\mathcal{G}_{t}^{k^{\prime},\epsilon}\\
&  =\widehat{\theta_{\epsilon}}\left(  t,\xi\right)  i\left(  \sigma_{k}%
\cdot\xi\right)  \frac{1}{2\epsilon}\int_{0}^{t}1_{\left[  (s-\epsilon
)_{+},s+\epsilon\right]  }\left(  r\right)  ds,
\end{align*}
where we have used (\ref{approximate derivative}).
\end{proof}

We come back to equation (\ref{F 2}). Using (\ref{basic formula}) we have%
\begin{align*}
\widehat{\theta_{\epsilon}}\left(  t,\xi\right)   &  =\widehat{\theta_{0}%
}\left(  \xi\right)  -\kappa\left\vert \xi\right\vert ^{2}\int_{0}%
^{t}\widehat{\theta_{\epsilon}}\left(  s,\xi\right)  ds+M_{t}\\
&  +i\sum_{k\in K}\left(  \sigma_{k}\cdot\xi\right)  \int_{0}^{t}\left\langle
D_{\cdot}^{\left(  k\right)  }\widehat{\theta_{\epsilon}}\left(  s,\xi\right)
,\frac{1}{2\epsilon}1_{\left[  (s-\epsilon)_{+},s+\epsilon\right]
}\right\rangle _{\mathcal{H}}ds,
\end{align*}
where $M$ is a mean zero process, and using (\ref{F Malliavin})
\begin{align*}
&  =\widehat{\theta_{0}}\left(  \xi\right)  -\kappa\left\vert \xi\right\vert
^{2}\int_{0}^{t}\widehat{\theta_{\epsilon}}\left(  s,\xi\right)  ds+M_{t}\\
&  -\sum_{k\in K}\left(  \sigma_{k}\cdot\xi\right)  ^{2}\int_{0}%
^{t}\widehat{\theta_{\epsilon}}\left(  s,\xi\right)  \frac{1}{\left(
2\epsilon\right)  ^{2}}\int_{0}^{s}\left\langle 1_{\left[  (r-\epsilon
)_{+},r+\epsilon\right]  },1_{\left[  (s-\epsilon)_{+},s+\epsilon\right]
}\right\rangle _{\mathcal{H}}drds.
\end{align*}
Recalling the definition of $\mathcal{V}_{\epsilon}$\ and taking expectation
we get for $e_{\epsilon}\left(  t,\xi\right)  =\mathbb{E}\left[
\widehat{\theta_{\epsilon}}\left(  t,\xi\right)  \right]  $%

\begin{equation}
e_{\epsilon}\left(  t,\xi\right)  =\widehat{\theta_{0}}\left(  \xi\right)
-\kappa\left\vert \xi\right\vert ^{2}\int_{0}^{t}e_{\epsilon}\left(
s,\xi\right)  ds-\sigma^{2}\left(  \xi\right)  \int_{0}^{t}e_{\epsilon}\left(
s,\xi\right)  d\mathcal{V}_{\epsilon}\left(  s\right) , \label{F 12}%
\end{equation}
where
\[
\sigma^{2}\left(  \xi\right)  :=\sum_{k\in K}\left(  \sigma_{k}\cdot
\xi\right)  ^{2}.
\]
The limit
\[
e\left(  t,\xi\right)  =\lim_{\epsilon\rightarrow0}e_{\epsilon}\left(
t,\xi\right)
\]
exists (it can be deduced by a stability argument on the differential equation
and Lemma \ref{lemma1}, but for shortness let us invoke here the explicit
formula (\ref{F explicit})). Taking the limit as $\epsilon\rightarrow0$ into
the previous equation, by Lemma \ref{lemma1} we get the result of the
following corollary.

\begin{corollary}
Suppose Assumption A. Then the function $e\left(  t,\xi\right)  $ satisfies
the closed form equation%
\[
e\left(  t,\xi\right)  =\widehat{\theta_{0}}\left(  \xi\right)  -\kappa
\left\vert \xi\right\vert ^{2}\int_{0}^{t}e\left(  s,\xi\right)  ds-\sigma
^{2}\left(  \xi\right)  \int_{0}^{t}e\left(  s,\xi\right)  d\gamma\left(
s\right) .
\]

\end{corollary}

Now we go back in physical space by inverse Fourier transform. Until now we
have assumed only Assumption A.

\begin{remark}
  \label{Remark about singularity}
  Without additional assumptions the Fourier
coefficients $e\left(  t,\xi\right)  $ could not have easy decay properties
for large $\xi$, in the case when the measure $d\gamma\left(  s\right)  $ has
a negative component, and as a consequence the inverse Fourier transform could
give us a true distribution, which solves in the distributional sense the
equation written below but should require a closer investigation, due to its
singularity. Similarly, if $d\gamma\left(  s\right)  /ds$ is well-defined, non
negative, but it is not bounded above, like in the case of FBM with $H<1/2$
where it is diverges at $s=0$, we are faced - in the inverse Fourier transform
- with a parabolic equation equation with singular second order coefficients,
which is uncommon and also requires special theory.
\end{remark}

\begin{remark}
\label{Remark about Physics} In addition, the present work has a precise
motivation from 2D inverse cascade turbulence and, in that framework, we
expect the large vortex structures being positively correlated in time, as it
is for $H>1/2$, not negatively as it is for $H<1/2$. Therefore we prefer to
assume $d\gamma\left(  s\right)  /ds$ bounded from above for reasons of
coherence with the purposes of this work. On the contrary, the case when
$d\gamma\left(  s\right)  $ has a negative component could correspond to
negative viscosity, a debated phenomenon for turbulent fluids \cite{Wirth},
perhaps also associated with the 2D inverse cascade. However, it must be
better understood and thus we postpone to future works.
\end{remark}

For the reasons highlighted in the previous two remarks, we introduce the
following additional assumption:

\textbf{Assumption B}

i) the measure $d\gamma\left(  s\right)  $ has a non-negative density
$d\gamma\left(  s\right)  /ds$ with respect to Lebesgue measure

ii) and there exists $C>0$ such that $d\gamma\left(  s\right)  /ds\leq C$ for
a.e. $s\geq0$.

We may call "regular" the case when Assumption B is satisfied and "singular"
the other case, which covers measures with negative components of viscosity
and unbounded positive viscosities.

Under Assumption B, $\left\vert e\left(  t,\xi\right)  \right\vert
\leq\left\vert \widehat{\theta_{0}}\left(  \xi\right)  \right\vert $, hence
the following.

\begin{corollary}
Suppose Assumptions A and B and $\theta_{0}\in L^{2}\left(  \mathbb{R}%
^{2}\right)  $. Then%
\[
\overline{\theta}\left(  t,x\right)  :=\int_{\mathbb{R}^{d}}e^{2\pi i\xi\cdot
x}e\left(  t,\xi\right)  d\xi
\]
has the property $\overline{\theta}\left(  t\right)  \in L^{2}\left(
\mathbb{R}^{2}\right)  $, it is weakly continuous in time in $L^{2}\left(
\mathbb{R}^{2}\right)  $ and it satisfies, in the sense of distribution,
\[
\overline{\theta}\left(  t,x\right)  =\theta_{0}\left(  x\right)  +\int%
_{0}^{t}\kappa\Delta\overline{\theta}\left(  s,x\right)  ds+\int_{0}%
^{t}\left(  \mathcal{L}\overline{\theta}\left(  s\right)  \right)  \left(
x\right)  d\gamma\left(  s\right) ,
\]
where $\mathcal{L}$ is the differential operator defined by (\ref{diff oper}).
\end{corollary}

The only part of the statement we have to clarify is the form of $\mathcal{L}%
$. Until now it is%

\[
\widehat{\left(  \mathcal{L}f\right)  }\left(  \xi\right)  =-\sigma^{2}\left(
\xi\right)  \widehat{f}\left(  \xi\right)  .
\]
Then%
\begin{align*}
\left(  \mathcal{L}f\right)  \left(  x\right)   &  =\sum_{k\in K}%
\int_{\mathbb{R}^{d}}e^{2\pi i\xi\cdot x}\left(  i\sigma_{k}\cdot\xi\right)
\left(  i\sigma_{k}\cdot\xi\right)  \widehat{f}\left(  \xi\right)  d\xi\\
&  =\sum_{k\in K}\left(  \sigma_{k}\cdot\nabla\right)  \left(  \sigma_{k}%
\cdot\nabla f\left(  x\right)  \right) \\
&  =\sum_{k\in K}\operatorname{div}\left(  \left(  \sigma_{k}\otimes\sigma
_{k}\right)  \nabla f\left(  x\right)  \right) \\
&  =\operatorname{div}\left(  Q\nabla f\left(  x\right)  \right)  ,
\end{align*}
where $Q=\sum_{k\in K}\left(  \sigma_{k}\otimes\sigma_{k}\right)  $.

\begin{remark}
Assume, for instance, that $N=2$, $K=\left\{  1,2\right\}  $, $e_{1},e_{1}$
canonical basis of $\mathbb{R}^{2}$,
\[
\sigma_{k}=\sqrt{\kappa_{T}}e_{1},
\]
where $\kappa_{T}>0$ is a constant (with the physical meaning of turbulent
kinetic energy). Then
\begin{align*}
\sigma^{2}\left(  \xi\right)   &  =\kappa_{T}\left\vert \xi\right\vert ^{2}\\
\left(  \mathcal{L}f\right)  \left(  x\right)   &  =\kappa_{T}\Delta f\left(
x\right)
\end{align*}
and the equations take the form%
\[
e\left(  t,\xi\right)  =\widehat{\theta_{0}}\left(  \xi\right)  -\kappa
\left\vert \xi\right\vert ^{2}\int_{0}^{t}e\left(  s,\xi\right)  ds-\kappa
_{T}\left\vert \xi\right\vert ^{2}\int_{0}^{t}e\left(  s,\xi\right)
d\gamma\left(  s\right)
\]%
\[
\overline{\theta}\left(  t,x\right)  =\theta_{0}\left(  x\right)  +\int%
_{0}^{t}\kappa\Delta\overline{\theta}\left(  s,x\right)  ds+\int_{0}^{t}%
\kappa_{T}\Delta\overline{\theta}\left(  s,x\right)  d\gamma\left(  s\right)
.
\]
The dissipation $\kappa\Delta\overline{\theta}\left(  s,x\right)  $ is
enhanced by the term $\kappa_{T}\Delta\overline{\theta}\left(  s,x\right)
d\gamma\left(  s\right)  $, on average. However, compared to the Brownian case
$H=1/2$, where%
\[
\overline{\theta}\left(  t,x\right)  =\theta_{0}\left(  x\right)  +\int%
_{0}^{t}\left(  \kappa+\kappa_{T}\right)  \Delta\overline{\theta}\left(
s,x\right)  ds,
\]
the case when $H>1/2$ is slower for short times, because $d\gamma\left(
s\right)  \sim s^{H-\frac{1}{2}}ds$ which is infinitesimal for small $s$.
Positively correlated noise decreases the dissipation power with respect to
the incorrelated case, which is constituted by Gaussian white noise.
\end{remark}

We go on investigating the variance in closed form.

\subsection{Variance-covariance of the solution}

In order to evaluate the variance of the solution one needs to understand the
covariance structure of $\left(  \widehat{\theta_{\epsilon}}\left(
t,\xi\right)  \right)  _{\xi\in\mathbb{R}^{2}}$. Indeed, if $\widehat{\theta
_{\epsilon}}$ is the solution of equation (\ref{F 1}) then%
\[
\theta_{\epsilon}\left(  t,x\right)  -\mathbb{E}\left[  \theta_{\epsilon
}\left(  t,x\right)  \right]  =\int e^{2\pi i\xi\cdot x}\left(
\widehat{\theta_{\epsilon}}\left(  t,\xi\right)  -e_{\epsilon}\left(
t,\xi\right)  \right)  d\xi
\]%
\begin{align*}
Var\left(  \theta_{\epsilon}\left(  t,x\right)  \right)   &  =\mathbb{E}%
\left[  \left(  \theta_{\epsilon}\left(  t,x\right)  -\mathbb{E}\left[
\theta_{\epsilon}\left(  t,x\right)  \right]  \right)  \overline{\left(
\theta_{\epsilon}\left(  t,x\right)  -\mathbb{E}\left[  \theta_{\epsilon
}\left(  t,x\right)  \right]  \right)  }\right] \\
&  =\int\int e^{2\pi i\left(  \xi-\eta\right)  \cdot x}C_{\epsilon}\left(
t,\xi,\eta\right)  d\xi d\eta,
\end{align*}
where $C_{\epsilon}\left(  t,\xi,\eta\right)  $ is the covariance function%
\[
C_{\epsilon}\left(  t,\xi,\eta\right)  =\mathbb{E}\left[  \left(  \left(
\widehat{\theta_{\epsilon}}\left(  t,\xi\right)  -e_{\epsilon}\left(
t,\xi\right)  \right)  \right)  \left(  \overline{\widehat{\theta_{\epsilon}%
}\left(  t,\eta\right)  -e_{\epsilon}\left(  t,\eta\right)  }\right)  \right]
.
\]
We come back to equations (\ref{F 2}) and (\ref{F 12}). We set%
\[
\widetilde{\theta}_{\epsilon}\left(  t,\xi\right)  =\widehat{\theta_{\epsilon
}}\left(  t,\xi\right)  -e_{\epsilon}\left(  t,\xi\right)
\]
and have%
\begin{align*}
\widetilde{\theta}_{\epsilon}\left(  t,\xi\right)   &  =-\kappa\left\vert
\xi\right\vert ^{2}\int_{0}^{t}\widetilde{\theta}_{\epsilon}\left(
s,\xi\right)  ds\\
&  +i\sum_{k\in K}\left(  \sigma_{k}\cdot\xi\right)  \int_{0}^{t}%
\widetilde{\theta}_{\epsilon}\left(  s,\xi\right)  \frac{G_{s+\epsilon}%
^{k}-G_{(s-\epsilon)_{+}}^{k}}{2\epsilon}ds\\
&  +i\sum_{k\in K}\left(  \sigma_{k}\cdot\xi\right)  \int_{0}^{t}e_{\epsilon
}\left(  s,\xi\right)  \frac{G_{s+\epsilon}^{k}-G_{(s-\epsilon)_{+}}^{k}%
}{2\epsilon}ds\\
&  -\sigma^{2}\left(  \xi\right)  \int_{0}^{t}e_{\epsilon}\left(
s,\xi\right)  d\mathcal{V}_{\epsilon}\left(  s\right) .
\end{align*}
Therefore%
\begin{align*}
&  \widetilde{\theta}_{\epsilon}\left(  t,\xi\right)  \overline
{\widetilde{\theta}_{\epsilon}\left(  t,\eta\right)  }\\
&  =\int_{0}^{t}\widetilde{\theta}_{\epsilon}\left(  ds,\xi\right)
\overline{\widetilde{\theta}_{\epsilon}\left(  s,\eta\right)  }+\int_{0}%
^{t}\widetilde{\theta}_{\epsilon}\left(  s,\xi\right)  \overline
{\widetilde{\theta}_{\epsilon}\left(  ds,\eta\right)  }%
\end{align*}%
\begin{align*}
&  =-\kappa\left(  \left\vert \xi\right\vert ^{2}+\left\vert \eta\right\vert
^{2}\right)  \int_{0}^{t}\widetilde{\theta}_{\epsilon}\left(  s,\xi\right)
\overline{\widetilde{\theta}_{\epsilon}\left(  s,\eta\right)  }ds\\
&  +i\sum_{k\in K}\left(  \sigma_{k}\cdot\left(  \xi-\eta\right)  \right)
\int_{0}^{t}\widetilde{\theta}_{\epsilon}\left(  s,\xi\right)  \overline
{\widetilde{\theta}_{\epsilon}\left(  s,\eta\right)  }\frac{G_{s+\epsilon}%
^{k}-G_{(s-\epsilon)_{+}}^{k}}{2\epsilon}ds\\
&  +i\sum_{k\in K}\left(  \sigma_{k}\cdot\xi\right)  \int_{0}^{t}e_{\epsilon
}\left(  s,\xi\right)  \overline{\widetilde{\theta}_{\epsilon}\left(
s,\eta\right)  }\frac{G_{s+\epsilon}^{k}-G_{(s-\epsilon)_{+}}^{k}}{2\epsilon
}ds\\
&  -i\sum_{k\in K}\left(  \sigma_{k}\cdot\eta\right)  \int_{0}^{t}%
\widetilde{\theta}_{\epsilon}\left(  s,\xi\right)  \overline{e_{\epsilon
}\left(  s,\eta\right)  }\frac{G_{s+\epsilon}^{k}-G_{(s-\epsilon)_{+}}^{k}%
}{2\epsilon}ds\\
&  -\sigma^{2}\left(  \xi\right)  \int_{0}^{t}e_{\epsilon}\left(
s,\xi\right)  \overline{\widetilde{\theta}_{\epsilon}\left(  s,\eta\right)
}d\mathcal{V}_{\epsilon}\left(  s\right) \\
&  -\sigma^{2}\left(  \eta\right)  \int_{0}^{t}\widetilde{\theta}_{\epsilon
}\left(  s,\xi\right)  \overline{e_{\epsilon}\left(  s,\eta\right)
}d\mathcal{V}_{\epsilon}\left(  s\right) .
\end{align*}
So setting%
\[
R_{\epsilon}\left(  t,\xi,\eta\right)  =\widetilde{\theta}_{\epsilon}\left(
t,\xi\right)  \overline{\widetilde{\theta}_{\epsilon}\left(  t,\eta\right)
},
\]
we get%
\begin{align*}
R_{\epsilon}\left(  t,\xi,\eta\right)   &  =-\kappa\left(  \left\vert
\xi\right\vert ^{2}+\left\vert \eta\right\vert ^{2}\right)  \int_{0}%
^{t}R_{\epsilon}\left(  s,\xi,\eta\right)  ds\\
&  +i\sum_{k\in K}\left(  \sigma_{k}\cdot\left(  \xi-\eta\right)  \right)
\int_{0}^{t}R_{\epsilon}\left(  s,\xi,\eta\right)  \frac{G_{s+\epsilon}%
^{k}-G_{(s-\epsilon)_{+}}^{k}}{2\epsilon}ds\\
&  +i\sum_{k\in K}\left(  \sigma_{k}\cdot\xi\right)  \int_{0}^{t}e_{\epsilon
}\left(  s,\xi\right)  \overline{\widetilde{\theta}_{\epsilon}\left(
s,\eta\right)  }\frac{G_{s+\epsilon}^{k}-G_{(s-\epsilon)_{+}}^{k}}{2\epsilon
}ds\\
&  -i\sum_{k\in K}\left(  \sigma_{k}\cdot\eta\right)  \int_{0}^{t}%
\widetilde{\theta}_{\epsilon}\left(  s,\xi\right)  \overline{e_{\epsilon
}\left(  s,\eta\right)  }\frac{G_{s+\epsilon}^{k}-G_{(s-\epsilon)_{+}}^{k}%
}{2\epsilon}ds\\
&  +M_{t},
\end{align*}
where $M$ is a mean zero process.

Now we need to express the three terms on the right-hand-side of this identity
which involve the noise by means of mean zero processes plus a trace. Let us
treat each one of them. Denoting again by $M$ a generic mean zero process, we
re-express the first one of the previous terms, using in particular
(\ref{basic formula}), (\ref{F Malliavin}) and (\ref{F 3}):%
\[
\int_{0}^{t}R_{\epsilon}\left(  s,\xi,\eta\right)  \frac{G_{s+\epsilon}%
^{k}-G_{(s-\epsilon)_{+}}^{k}}{2\epsilon}ds=M_{t}+\int_{0}^{t}\left\langle
D_{\cdot}^{\left(  k\right)  }R_{\epsilon}\left(  s,\xi,\eta\right)  ,\frac
{1}{2\epsilon}1_{\left[  (s-\epsilon)_{+},s+\epsilon\right]  }\right\rangle
_{\mathcal{H}}ds.
\]
From (\ref{F Malliavin}),
\begin{align*}
D_{r}^{\left(  k\right)  }R_{\epsilon}\left(  t,\xi,\eta\right)   &
=D_{r}^{\left(  k\right)  }\widehat{\theta_{\epsilon}}\left(  t,\xi\right)
\cdot\overline{\widetilde{\theta}_{\epsilon}\left(  t,\eta\right)
}+\widetilde{\theta}_{\epsilon}\left(  t,\xi\right)  \cdot\overline
{D_{r}^{\left(  k\right)  }\widehat{\theta_{\epsilon}}\left(  t,\eta\right)
}\\
&  =\widehat{\theta_{\epsilon}}\left(  t,\xi\right)  \overline
{\widetilde{\theta}_{\epsilon}\left(  t,\eta\right)  }i\left(  \sigma_{k}%
\cdot\xi\right)  \frac{1}{2\epsilon}\int_{0}^{t}1_{\left[  (s-\epsilon
)_{+},s+\epsilon\right]  }\left(  r\right)  ds\\
&  -\widetilde{\theta}_{\epsilon}\left(  t,\xi\right)  \overline
{\widehat{\theta_{\epsilon}}\left(  t,\eta\right)  }i\left(  \sigma_{k}%
\cdot\eta\right)  \frac{1}{2\epsilon}\int_{0}^{t}1_{\left[  (s-\epsilon
)_{+},s+\epsilon\right]  }\left(  r\right)  ds\\
&  =M_{t}+R_{\epsilon}\left(  t,\xi,\eta\right)  i\left(  \sigma_{k}%
\cdot\left(  \xi-\eta\right)  \right)  \frac{1}{2\epsilon}\int_{0}%
^{t}1_{\left[  (s-\epsilon)_{+},s+\epsilon\right]  }\left(  r\right)  ds
\end{align*}
and thus%
\[
\int_{0}^{t}\left\langle D_{\cdot}^{\left(  k\right)  }R_{\epsilon}\left(
s,\xi,\eta\right)  ,1_{\left[  (s-\epsilon)_{+},s+\epsilon\right]
}\right\rangle _{\mathcal{H}}ds=M_{t}+i\left(  \sigma_{k}\cdot\left(  \xi
-\eta\right)  \right)  \int_{0}^{t}R_{\epsilon}\left(  s,\xi,\eta\right)
d\mathcal{V}_{\epsilon}\left(  s\right) .
\]
Concerning the third one of the previous terms we have%

\begin{align*}
&  \int_{0}^{t}\widetilde{\theta}_{\epsilon}\left(  s,\xi\right)
\overline{e_{\epsilon}\left(  s,\eta\right)  }\frac{G_{s+\epsilon}%
^{k}-G_{(s-\epsilon)_{+}}^{k}}{2\epsilon}ds\\
&  =M_{t}+\int_{0}^{t}\overline{e_{\epsilon}\left(  s,\eta\right)
}\left\langle D_{\cdot}^{\left(  k\right)  }\widehat{\theta_{\epsilon}}\left(
s,\xi\right)  ,\frac{1}{2\epsilon}1_{\left[  (s-\epsilon)_{+},s+\epsilon
\right]  }\right\rangle _{\mathcal{H}}ds\\
&  =M_{t}+i\left(  \sigma_{k}\cdot\xi\right)  \int_{0}^{t}\widehat{\theta
_{\epsilon}}\left(  s,\xi\right)  \overline{e_{\epsilon}\left(  s,\eta\right)
}\frac{1}{\left(  2\epsilon\right)  ^{2}}\int_{0}^{s}\left\langle 1_{\left[
(r-\epsilon)_{+},r+\epsilon\right]  },1_{\left[  (s-\epsilon)_{+}%
,s+\epsilon\right]  }\right\rangle _{\mathcal{H}}drds\\
&  =M_{t}+i\left(  \sigma_{k}\cdot\xi\right)  \int_{0}^{t}\widehat{\theta
_{\epsilon}}\left(  s,\xi\right)  \overline{e_{\epsilon}\left(  s,\eta\right)
}d\mathcal{V}_{\epsilon}\left(  s\right),
\end{align*}
hence%
\begin{align*}
&  -i\sum_{k\in K}\left(  \sigma_{k}\cdot\eta\right)  \int_{0}^{t}%
\widetilde{\theta}_{\epsilon}\left(  s,\xi\right)  \overline{e_{\epsilon
}\left(  s,\eta\right)  }\frac{G_{s+\epsilon}^{k}-G_{(s-\epsilon)_{+}}^{k}%
}{2\epsilon}ds\\
&  =M_{t}+\sum_{k\in K}\left(  \sigma_{k}\cdot\xi\right)  \left(  \sigma
_{k}\cdot\eta\right)  \int_{0}^{t}\widehat{\theta}_{\epsilon}\left(
s,\xi\right)  \overline{e_{\epsilon}\left(  s,\eta\right)  }d\mathcal{V}%
_{\epsilon}\left(  s\right)  .
\end{align*}
Similarly, concerning the second one of the terms,
\[
\int_{0}^{t}e_{\epsilon}\left(  s,\xi\right)  \overline{\widetilde{\theta
}_{\epsilon}\left(  s,\eta\right)  }\frac{G_{s+\epsilon}^{k}-G_{(s-\epsilon
)_{+}}^{k}}{2\epsilon}ds=M_{t}-i\left(  \sigma_{k}\cdot\eta\right)  \int%
_{0}^{t}e_{\epsilon}\left(  s,\xi\right)  \overline{\widehat{\theta}%
_{\epsilon}\left(  s,\eta\right)  }d\mathcal{V}_{\epsilon}\left(  s\right)
\]
and therefore%
\begin{align*}
&  i\sum_{k\in K}\left(  \sigma_{k}\cdot\xi\right)  \int_{0}^{t}e_{\epsilon
}\left(  s,\xi\right)  \overline{\widetilde{\theta}_{\epsilon}\left(
s,\eta\right)  }\frac{G_{s+\epsilon}^{k}-G_{(s-\epsilon)_{+}}^{k}}{2\epsilon
}ds\\
&  =M_{t}+\sum_{k\in K}\left(  \sigma_{k}\cdot\xi\right)  \left(  \sigma
_{k}\cdot\eta\right)  \int_{0}^{t}e_{\epsilon}\left(  s,\xi\right)
\overline{\widehat{\theta_{\epsilon}}\left(  s,\eta\right)  }d\mathcal{V}%
_{\epsilon}\left(  s\right) .
\end{align*}
Summarizing the previous identities, we get%
\begin{align*}
R_{\epsilon}\left(  t,\xi,\eta\right)   &  =-\kappa\left(  \left\vert
\xi\right\vert ^{2}+\left\vert \eta\right\vert ^{2}\right)  \int_{0}%
^{t}R_{\epsilon}\left(  s,\xi,\eta\right)  ds\\
&  -\sum_{k\in K}\left(  \sigma_{k}\cdot\left(  \xi-\eta\right)  \right)
^{2}\int_{0}^{t}R_{\epsilon}\left(  s,\xi,\eta\right)  d\mathcal{V}_{\epsilon
}\left(  s\right) \\
&  +\sum_{k\in K}\left(  \sigma_{k}\cdot\xi\right)  \left(  \sigma_{k}%
\cdot\eta\right)  \int_{0}^{t}e_{\epsilon}\left(  s,\xi\right)  \overline
{\widehat{\theta_{\epsilon}}\left(  s,\eta\right)  }d\mathcal{V}_{\epsilon
}\left(  s\right) \\
&  +\sum_{k\in K}\left(  \sigma_{k}\cdot\xi\right)  \left(  \sigma_{k}%
\cdot\eta\right)  \int_{0}^{t}\widehat{\theta_{\epsilon}}\left(  s,\xi\right)
\overline{e_{\epsilon}\left(  s,\eta\right)  }d\mathcal{V}_{\epsilon}\left(
s\right) \\
&  +M_{t}.
\end{align*}
Taking expectation (recall $C_{\epsilon}\left(  t,\xi,\eta\right)
=\mathbb{E}\left[  R_{\epsilon}\left(  t,\xi,\eta\right)  \right]  $), we get%
\begin{align*}
C_{\epsilon}\left(  t,\xi,\eta\right)   &  =-\kappa\left(  \left\vert
\xi\right\vert ^{2}+\left\vert \eta\right\vert ^{2}\right)  \int_{0}%
^{t}C_{\epsilon}\left(  s,\xi,\eta\right)  ds\\
&  -\sum_{k\in K}\left(  \sigma_{k}\cdot\left(  \xi-\eta\right)  \right)
^{2}\int_{0}^{t}C_{\epsilon}\left(  s,\xi,\eta\right)  d\mathcal{V}_{\epsilon
}\left(  s\right) \\
&  +2\sum_{k\in K}\left(  \sigma_{k}\cdot\xi\right)  \left(  \sigma_{k}%
\cdot\eta\right)  \int_{0}^{t}e_{\epsilon}\left(  s,\xi\right)  \overline
{e_{\epsilon}\left(  s,\eta\right)  }d\mathcal{V}_{\epsilon}\left(  s\right)
.
\end{align*}
Let us introduce the notation%
\[
\rho\left(  \xi,\eta\right)  :=\sum_{k\in K}\left(  \sigma_{k}\cdot\xi\right)
\left(  \sigma_{k}\cdot\eta\right)  .
\]
Taking the limit (it existence is like above for $e\left(  t,\xi\right)  $)
\[
C\left(  t,\xi,\eta\right)  :=\lim_{\epsilon\rightarrow0}C_{\epsilon}\left(
t,\xi,\eta\right)  ,
\]
from Lemma \ref{lemma1}, we establish the following.

\begin{proposition}
\label{Proposition variance Fourier}Suppose Assumption A. Then the function
$C\left(  t,\xi,\eta\right)  $ satisfies, together with $e$, the identity
\begin{align*}
C\left(  t,\xi,\eta\right)   &  =-\kappa\left(  \left\vert \xi\right\vert
^{2}+\left\vert \eta\right\vert ^{2}\right)  \int_{0}^{t}C\left(  s,\xi
,\eta\right)  ds\\
&  -\sigma^{2}\left(  \xi-\eta\right)  \int_{0}^{t}C\left(  s,\xi,\eta\right)
d\gamma\left(  s\right) \\
&  +2\rho\left(  \xi,\eta\right)  \int_{0}^{t}e\left(  s,\xi\right)
\overline{e\left(  s,\eta\right)  }d\gamma\left(  s\right) .
\end{align*}

\end{proposition}

Let us go back to the computation of the function
\[
V\left(  t,x\right)  =\int\int e^{2\pi i\left(  \xi-\eta\right)  \cdot
x}C\left(  t,\xi,\eta\right)  d\xi d\eta,
\]
which is limit of $Var\left(  \theta_{\epsilon}\left(  t,x\right)  \right)  $.

\begin{proposition}
Suppose Assumptions A and B. Then, in the sense of distributions,%
\[
V\left(  t,x\right)  =\int_{0}^{t}2\kappa\Delta V\left(  s,x\right)
ds+\int_{0}^{t}\left(  \mathcal{L}V\left(  s\right)  \right)  \left(
x\right)  d\gamma\left(  s\right)  +2\sum_{k\in K}\int_{0}^{t}\left(  \left(
\sigma_{k}\cdot\nabla\right)  \overline{\theta}\left(  s,x\right)  \right)
^{2}d\gamma\left(  s\right) .
\]

\end{proposition}

\begin{proof}
The proof proceeds by applying $\int\int e^{2\pi i\left(  \xi-\eta\right)
\cdot x}\cdot\cdot\cdot d\xi d\eta$ to each term of the identity of
Proposition \ref{Proposition variance Fourier}. Its application to the first
term on the right-hand-side gives us $\int_{0}^{t}2\Delta V\left(  s,x\right)
ds$ because%
\begin{align*}
&  \int\int e^{2\pi i\left(  \xi-\eta\right)  \cdot x}\left\vert
\xi\right\vert ^{2}C\left(  s,\xi,\eta\right)  d\xi d\eta\\
&  =\int e^{2\pi i\xi\cdot x}\left\vert \xi\right\vert ^{2}\left(  \int
e^{-2\pi i\eta\cdot x}C\left(  s,\xi,\eta\right)  d\eta\right)  d\xi\\
&  =-\Delta\int e^{2\pi i\xi\cdot x}\left(  \int e^{-2\pi i\eta\cdot
x}C\left(  s,\xi,\eta\right)  d\eta\right)  d\xi\\
&  =-\Delta V\left(  t,x\right)
\end{align*}
and a similar computation on the conjugate holds for the other term.

Let us come to the second term on the right-hand-side of the identity of
Proposition \ref{Proposition variance Fourier}. To make the computation more
transparent, let us write $V\left(  t,x\right)  $ in Fourier form as
\begin{align*}
V\left(  t,x\right)   &  =\int\int e^{2\pi i\left(  \xi-\eta\right)  \cdot
x}C\left(  t,\xi,\eta\right)  d\xi d\eta\\
&  \overset{\xi^{\prime}=\xi-\eta}{=}\int\left(  \int e^{2\pi i\xi^{\prime
}\cdot x}C\left(  s,\xi^{\prime}+\eta,\eta\right)  d\xi^{\prime}\right)
d\eta\\
&  =\int e^{2\pi i\xi^{\prime}\cdot x}\left(  \int C\left(  s,\xi^{\prime
}+\eta,\eta\right)  d\eta\right)  d\xi^{\prime}=\int e^{2\pi i\xi\cdot
x}\widehat{V}\left(  t,\xi\right)  d\xi,
\end{align*}
where%
\[
\widehat{V}\left(  t,\xi\right)  =\int C\left(  s,\xi+\eta,\eta\right)
d\eta.
\]
Then%
\begin{align*}
\widehat{\left(  \mathcal{L}V\left(  t\right)  \right)  }\left(  \xi\right)
&  =\sigma^{2}\left(  \xi\right)  \widehat{V}\left(  t,\xi\right) \\
&  =\sigma^{2}\left(  \xi\right)  \int C\left(  s,\xi+\eta,\eta\right)  d\eta
\end{align*}%
\begin{align*}
&  \left(  \mathcal{L}V\left(  s\right)  \right)  \left(  x\right) \\
&  =\int e^{2\pi i\xi\cdot x}\widehat{\left(  \mathcal{L}V\left(  s\right)
\right)  }\left(  \xi\right)  d\xi\\
&  =\int e^{2\pi i\xi\cdot x}\sigma^{2}\left(  \xi\right)  \int C\left(
s,\xi+\eta,\eta\right)  d\eta d\xi\\
&  \overset{\xi^{\prime}=\xi+\eta}{=}\int\int e^{2\pi i\left(  \xi^{\prime
}-\eta\right)  \cdot x}\sigma^{2}\left(  \xi^{\prime}-\eta\right)  C\left(
s,\xi^{\prime},\eta\right)  d\xi^{\prime}d\eta.
\end{align*}
Thus also the second term is checked.

Finally, let us treat the third term on the right-hand-side of the identity of
Proposition \ref{Proposition variance Fourier}. Here we simply have
\begin{align*}
&  \int\int e^{2\pi i\left(  \xi-\eta\right)  \cdot x}\left(  \sigma_{k}%
\cdot\xi\right)  \left(  \sigma_{k}\cdot\eta\right)  e\left(  s,\xi\right)
\overline{e\left(  s,\eta\right)  }d\xi d\eta\\
&  =\left(  \int e^{2\pi i\xi\cdot x}i\left(  \sigma_{k}\cdot\xi\right)
e\left(  s,\xi\right)  d\xi\right)  \left(  \overline{\int e^{2\pi i\eta\cdot
x}i\left(  \sigma_{k}\cdot\eta\right)  e\left(  s,\eta\right)  d\eta}\right)
\\
&  =\left(  \left(  \sigma_{k}\cdot\nabla\right)  \overline{\theta}\left(
s,x\right)  \right)  ^{2},
\end{align*}
which leads to the claimed identity.
\end{proof}

\section{The non commutative case\label{Section non commut}}

The case when the stochastic transport terms do not have constant-in-space
coefficients and do not commute between themselves and with the Laplacian, is
admittedly very difficult and still obscure, from the viewpoint of theoretical
quantitative results. We may only present two subsections with side remarks on
this topic.

The first subsection idealizes a turbulent 2D fluid undergoing inverse cascade
by prescribing two families of stochastic transport terms:\ a
small-space-scale component modeling the smallest turbulent scales,
maintained, white noise (namely uncorrelated) in time, and a
larger-space-scale component modeling the larger structures which appear and
disappear by inverse cascade. The latter are constant in space, idealization
of their relative size with respect to the smallest ones, and correlated in
time. The result we prove is that the smaller scales produce the effect
predicted by the Boussinesq hypothesis, while the larger ones are maintained
in their form. The system then reduces to the commutative case.

The second subsection is only aimed to explain as clearly as possible the
technical difficulty arising in a truly non-commutative case. A commutator
appears which spoils the simple link with the mean field equation found in
Section \ref{Section Commutative}. Nevertheless, a link up to a remainder
exists and could be important in future investigations.

\subsection{Two-scale system and reduction to the commutative
case\label{Section Galeati}}

Consider a more complete fluid dynamic model than the one introduced in
Section \ref{Section Intro}, further parametrized by a parameter
$N\in\mathbb{N}$, i.e.%

\begin{align*}
\partial_{t}\theta_{\epsilon,N}\left(  t,x\right)   &  =\kappa\Delta
\theta_{\epsilon,N}\left(  t,x\right) \\
&  +\mathcal{L}_{N}^{0}\theta_{\epsilon,N}\left(  t,x\right)  +\sum_{j\in
J_{N}}\left(  v_{j,N}\left(  x\right)  \cdot\nabla\right)  \theta_{\epsilon
,N}\left(  t,x\right)  \frac{dW_{t}^{j}}{dt}\\
&  +\sum_{k\in K}\left(  \sigma_{k}\left(  x\right)  \cdot\nabla\right)
\theta_{\epsilon,N}\left(  t,x\right)  \frac{d\mathcal{G}_{t}^{k,\epsilon}%
}{dt}\\
\theta_{\epsilon,N}|_{t=0}  &  =\theta_{0},
\end{align*}
where now $\sigma_{k}\left(  x\right)  $ are smooth divergence free fields,
the finite index set $J_{N}$ may vary with $N$, the vector fields
$v_{j,N}\left(  x\right)  $ too, as well as the associated covariance function
$Q_{N}^{0}\left(  x,y\right)  $ defined as%
\[
Q_{N}^{0}\left(  x,y\right)  =\sum_{j\in J_{N}}v_{j,N}\left(  x\right)
\otimes v_{j,N}\left(  y\right)
\]
covariance operator $\mathbb{Q}_{N}^{0}$ on vector fields $v,w\in L^{2}\left(
\mathbb{R}^{2}\right)  $ defined as%
\[
\left\langle \mathbb{Q}_{N}^{0}v,w\right\rangle _{_{L^{2}}}=\int\int w\left(
x\right)  ^{T}\cdot Q_{N}^{0}\left(  x,y\right)  \cdot v\left(  y\right)
dxdy
\]
and differential operator $\mathcal{L}_{N}^{0}$ defined as%
\[
\left(  \mathcal{L}_{N}^{0}f\right)  \left(  x\right)  =\operatorname{div}%
\left(  Q_{N}^{0}\left(  x,x\right)  \nabla f\left(  x\right)  \right) .
\]
The sum%
\[
\mathcal{L}_{N}^{0}\theta_{\epsilon,N}\left(  t,x\right)  +\sum_{j\in J_{N}%
}\left(  v_{j,N}\left(  x\right)  \cdot\nabla\right)  \theta_{\epsilon
,N}\left(  t,x\right)  \frac{dW_{t}^{j}}{dt}%
\]
stands for the It\^{o} formulation (easier to define) of the Stratonovich
integral
\[
\sum_{j\in J_{N}}\left(  v_{j,N}\left(  x\right)  \cdot\nabla\right)
\theta_{\epsilon,N}\left(  t,x\right)  \circ\frac{dW_{t}^{j}}{dt}.
\]
The main aim of this section is proving that, under suitable assumptions, we
may reduce the model to the commutative case. This requires that the It\^{o}
integrals go to zero and that the corrector goes to $\kappa_{T}\Delta
\theta_{\epsilon,N}\left(  t,x\right)  $. Since we want to interpret
rigorously the equation in mild form, in order to reduce details, we assume
that the diagonal $Q_{N}^{0}\left(  x,x\right)  $ is independent of $N$ and
already equal to $\kappa_{T}Id$, so
\[
Q_{N}^{0}\left(  x,x\right)  =\kappa_{T}Id,
\]
for every $N\in\mathbb{N}$ and $x\in\mathbb{R}^{2}$. Moreover, we assume that
$\kappa+\kappa_{T}>0$. Let $A$ be the infinitesimal generator of analytic
semigroup in $L^{2}\left(  \mathbb{R}^{2}\right)  $ (see \cite{Pa}, Chapter
7), defined on $W^{2,2}\left(  \mathbb{R}^{2}\right)  $ as%
\[
\left(  Af\right)  \left(  x\right)  =\left(  \kappa+\kappa_{T}\right)  \Delta
f\left(  x\right)  .
\]
Since we assume independence of $\left(  W^{j};j\in J_{N}\right)  $ from
$\left(  \mathcal{G}_{t}^{k,\epsilon};k\in K\right)  $ and the analysis in
this section is pathwise with respect to $\left(  \mathcal{G}_{t}^{k,\epsilon
};k\in K\right)  $, we replace the above equation by
\begin{align*}
d\theta_{\epsilon,N}\left(  t\right)   &  =\left(  A\theta_{\epsilon,N}\left(
t\right)  +\left(  v\left(  t\right)  \cdot\nabla\right)  \theta_{\epsilon
,N}\left(  t\right)  \right)  dt+\sum_{j\in J_{N}}\left(  v_{j,N}\cdot
\nabla\right)  \theta_{\epsilon,N}\left(  t\right)  dW_{t}^{j}\\
\theta_{\epsilon,N}|_{t=0}  &  =\theta_{0},
\end{align*}
where $v\left(  t,x\right)  $ is a single path of $\sum_{k\in K}\sigma
_{k}\left(  x\right)  \frac{d\mathcal{G}_{t}^{k,\epsilon}}{dt}$. This
equation, when $\theta_{0}\in L^{2}\left(  \mathbb{R}^{2}\right)  $, can be
solved, in mild form%
\begin{align*}
\theta_{\epsilon,N}\left(  t\right)   &  =e^{tA}\theta_{0}+\sum_{j\in J_{N}%
}\int_{0}^{t}e^{\left(  t-s\right)  A}\left(  v_{j,N}\cdot\nabla\right)
\theta_{\epsilon,N}\left(  s\right)  dW_{s}^{j}\\
&  +\int_{0}^{t}e^{\left(  t-s\right)  A}\left(  v\left(  s\right)
\cdot\nabla\right)  \theta_{\epsilon,N}\left(  s\right)  ds,
\end{align*}
as in the case $v\left(  t\right)  =0$ as in \cite{FlaLuongo2}, Chapter 3.
Here $e^{tA}$, $t\geq0$, denotes the analytic semigroup generated by $A$ on
$L^{2}\left(  \mathbb{R}^{2}\right)  $. The solution is an adapted process
with paths of class%
\[
\theta_{\epsilon,N}\in C\left(  \left[  0,T\right]  ;L^{2}\left(
\mathbb{R}^{2}\right)  \right)  \cap L^{2}\left(  0,T;W^{1,2}\left(
\mathbb{R}^{2}\right)  \right)  ,
\]
it satisfies a.s.%
\[
\sup_{t\in\left[  0,T\right]  }\left\Vert \theta_{\epsilon,N}\left(  t\right)
\right\Vert _{L^{2}\left(  \mathbb{R}^{2}\right)  }^{2}+\kappa\int_{0}%
^{T}\left\Vert \nabla\theta_{\epsilon,N}\left(  s\right)  \right\Vert
_{L^{2}\left(  \mathbb{R}^{2}\right)  }^{2}ds\leq\left\Vert \theta
_{0}\right\Vert _{L^{2}\left(  \mathbb{R}^{2}\right)  }^{2}.
\]
Moreover, it satisfies the maximum principle (\cite{FlaLuongo2}, Chapter 3;
see also \cite{FlaLuongo1})
\begin{equation}
\sup_{t\in\left[  0,T\right]  }\left\Vert \theta_{\epsilon,N}\left(  t\right)
\right\Vert _{L^{\infty}\left(  \mathbb{R}^{2}\right)  }\leq\left\Vert
\theta_{0}\right\Vert _{L^{\infty}\left(  \mathbb{R}^{2}\right)}.
\label{maximum prunciple}%
\end{equation}
However, with some easy technical work, we may also interpret the equation in
the alternative mild form
\begin{equation}
\theta_{\epsilon,N}\left(  t\right)  =U\left(  t,0\right)  \theta_{0}%
+\sum_{j\in J_{N}}\int_{0}^{t}U\left(  t,s\right)  \left(  v_{j,N}\cdot
\nabla\right)  \theta_{\epsilon,N}\left(  s\right)  dW_{s}^{j}, \label{mild U}
\end{equation}
where $U\left(  t,s\right)  $ is the evolution operator defined as follows.
For every $s\geq0$, consider the deterministic equation%
\begin{align*}
u\left(  t\right)   &  =e^{\left(  t-s\right)  A}u_{0}+\int_{s}^{t}e^{\left(
t-r\right)  A}\left(  v\left(  r\right)  \cdot\nabla\right)  u\left(
r\right)  dr\\
\text{for }t  &  \in\lbrack s,\infty).
\end{align*}
For every $u_{0}\in L^{2}\left(  \mathbb{R}^{2}\right)  $ and $T>s$, let
\[
u\in C\left(  \left[  s,T\right]  ;L^{2}\left(  \mathbb{R}^{2}\right)
\right)  \cap L^{2}\left(  s,T;W^{1,2}\left(  \mathbb{R}^{2}\right)  \right)
\]
be its unique solution. Then we set
\[
U\left(  t,s\right)  u_{0}=u\left(  t\right)  ,
\]
for $t\in\left[  s,T\right]  $ and extend to all $t$ in an obvious way using
the uniqueness. We construct a family of bounded linear operators $\left\{
U\left(  t,s\right)  ;0\leq s\leq t\right\}  $ on $L^{2}\left(  \mathbb{R}%
^{2}\right)  $. With a little work one can show that $\left(  t,s\right)
\longmapsto U\left(  t,s\right)  u_{0}$ is continuous, for every $u_{0}\in
L^{2}\left(  \mathbb{R}^{2}\right)  $ and that the mild formulation
(\ref{mild U}) based on $U\left(  t,s\right)  $ holds true. Finally, it is
easy to see that $U^{\ast}\left(  t,s\right)  $ is the analogous evolution
operator associated to the equation%
\begin{align*}
z\left(  t\right)   &  =e^{\left(  t-s\right)  A}z_{0}-\int_{s}^{t}e^{\left(
t-r\right)  A}\left(  v\left(  r\right)  \cdot\nabla\right)  z\left(
r\right)  dr\\
\text{for }t  &  \in\lbrack s,\infty).
\end{align*}
In particular it satisfies the inequality%
\begin{equation}
\sup_{t\in\left[  s,T\right]  }\left\Vert U^{\ast}\left(  t,s\right)
\phi\right\Vert _{L^{2}\left(  \mathbb{R}^{2}\right)  }^{2}+\kappa\int_{s}%
^{T}\left\Vert \nabla U^{\ast}\left(  t,s\right)  \phi\right\Vert
_{L^{2}\left(  \mathbb{R}^{2}\right)  }^{2}ds\leq\left\Vert \phi\right\Vert
_{L^{2}\left(  \mathbb{R}^{2}\right)  }^{2}. \label{U star}%
\end{equation}
Consider also the reduced problem%
\begin{align*}
\partial_{t}\theta_{\epsilon}\left(  t,x\right)   &  =\left(  \kappa
+\kappa_{T}\right)  \Delta\theta_{\epsilon}\left(  t,x\right)  +\sum_{k\in
K}\left(  \sigma_{k}\cdot\nabla\right)  \theta_{\epsilon}\left(  t,x\right)
\frac{d\mathcal{G}_{t}^{k,\epsilon}}{dt}\\
\theta_{\epsilon}|_{t=0}  &  =\theta_{0}.
\end{align*}
We simply have%
\[
\theta_{\epsilon}\left(  t\right)  =U\left(  t,0\right)  \theta_{0}.
\]

\begin{theorem}
We have%
\[
\mathbb{E}\left[  \left\langle \theta_{\epsilon,N}\left(  t\right)
-\theta_{\epsilon}\left(  t\right)  ,\phi\right\rangle ^{2}\right]  \leq
T\left\Vert \mathbb{Q}_{N}^{0}\right\Vert _{L^{2}\rightarrow L^{2}}\left\Vert
\theta_{0}\right\Vert _{L^{\infty}}^{2}\left\Vert \phi\right\Vert _{L^{2}}^{2}%
\]
for every $\phi\in L^{2}\left(  \mathbb{R}^{2}\right)  $ and $t\in\left[
0,T\right]  $. Therefore
\[
\lim_{N\rightarrow\infty}\sup_{t\in\left[  0,T\right]  }\mathbb{E}\left[
\left\langle \theta_{\epsilon,N}\left(  t\right)  -\theta_{\epsilon}\left(
t\right)  ,\phi\right\rangle ^{2}\right]  =0,
\]
if $\lim_{N\rightarrow\infty}\left\Vert \mathbb{Q}_{N}^{0}\right\Vert
_{L^{2}\rightarrow L^{2}}=0$.
\end{theorem}

\begin{proof}
One has%
\[
\left\langle \theta_{\epsilon,N}\left(  t\right)  -\theta_{\epsilon}\left(
t\right)  ,\phi\right\rangle =\sum_{j\in J_{N}}\int_{0}^{t}\left\langle
U\left(  t,s\right)  \left(  v_{j,N}\cdot\nabla\right)  \theta_{\epsilon
,N}\left(  s\right)  ,\phi\right\rangle dW_{s}^{j}%
\]%
\begin{align*}
\mathbb{E}\left[  \left\langle \theta_{\epsilon,N}\left(  t\right)
-\theta_{\epsilon}\left(  t\right)  ,\phi\right\rangle ^{2}\right]   &
=\sum_{j\in J_{N}}\mathbb{E}\int_{0}^{t}\left\langle U\left(  t,s\right)
\left(  v_{j,N}\cdot\nabla\right)  \theta_{\epsilon,N}\left(  s\right)
,\phi\right\rangle ^{2}ds\\
&  =\sum_{j\in J_{N}}\mathbb{E}\int_{0}^{t}\left\langle \theta_{\epsilon
,N}\left(  s\right)  ,\left(  v_{j,N}\cdot\nabla\right)  U^{\ast}\left(
t,s\right)  \phi\right\rangle ^{2}ds.
\end{align*}
Now, called $g_{t}\left(  s\right)  :=U^{\ast}\left(  t,s\right)  \phi$ for
shortness,%
\begin{align*}
&  \sum_{j\in J_{N}}\left\langle \theta_{\epsilon,N}\left(  s\right)  ,\left(
v_{j,N}\cdot\nabla\right)  g_{t}\left(  s\right)  \right\rangle ^{2}\\
&  =\sum_{j\in J_{N}}\int\int\theta_{\epsilon,N}\left(  s,x\right)  \left(
v_{j,N}\left(  x\right)  \cdot\nabla\right)  g_{t}\left(  s,x\right)
\theta_{\epsilon,N}\left(  s,y\right)  \left(  v_{j,N}\left(  y\right)
\cdot\nabla\right)  g_{t}\left(  s,y\right)  dxdy
\end{align*}%
\begin{align*}
&  =\sum_{\alpha,\beta=1}^{2}\sum_{j\in J_{N}}\int\int\theta_{\epsilon
,N}\left(  s,x\right)  v_{j,N}^{\alpha}\left(  x\right)  \partial_{\alpha
}g_{t}\left(  s,x\right)  \theta_{\epsilon,N}\left(  s,y\right)
v_{j,N}^{\beta}\left(  y\right)  \partial_{\beta}g_{t}\left(  s,y\right)
dxdy\\
&  =\sum_{\alpha,\beta=1}^{2}\int\int\theta_{\epsilon,N}\left(  s,x\right)
\left(  \sum_{j\in J_{N}}v_{j,N}^{\alpha}\left(  x\right)  v_{j,N}^{\beta
}\left(  y\right)  \right)  \partial_{\alpha}g_{t}\left(  s,x\right)
\theta_{\epsilon,N}\left(  s,y\right)  \partial_{\beta}g_{t}\left(
s,y\right)  dxdy
\end{align*}%
\begin{align*}
&  =\sum_{i,j=1}^{2}\int\int\theta_{\epsilon,N}\left(  s,x\right)
Q_{N}^{0,\alpha,\beta}\left(  x,y\right)  \partial_{\alpha}g_{t}\left(
s,x\right)  \theta_{\epsilon,N}\left(  s,y\right)  \partial_{\beta}%
g_{t}\left(  s,y\right)  dxdy\\
&  =\int\int\theta_{\epsilon,N}\left(  s,x\right)  \nabla g_{t}\left(
s,x\right)  ^{T}\cdot Q_{N}^{0}\left(  x,y\right)  \cdot\nabla g_{t}\left(
s,y\right)  \theta_{\epsilon,N}\left(  s,y\right)  dxdy
\end{align*}%
\begin{align*}
&  =\left\langle \mathbb{Q}_{N}^{0}\nabla g_{t}\left(  s\right)
\theta_{\epsilon,N}\left(  s\right)  ,\nabla g_{t}\left(  s\right)
\theta_{\epsilon,N}\left(  s\right)  \right\rangle \\
&  \leq\left\Vert \mathbb{Q}_{N}^{0}\right\Vert _{L^{2}\rightarrow L^{2}%
}\left\Vert \theta_{\epsilon,N}\left(  s\right)  \right\Vert _{L^{\infty}}%
^{2}\left\Vert \nabla U^{\ast}\left(  t,s\right)  \phi\right\Vert _{L^{2}}%
^{2}\\
&  \leq\left\Vert \mathbb{Q}_{N}^{0}\right\Vert _{L^{2}\rightarrow L^{2}%
}\left\Vert \theta_{0}\right\Vert _{L^{\infty}}^{2}\left\Vert \phi\right\Vert
_{L^{2}}^{2},
\end{align*}
by (\ref{maximum prunciple}) and (\ref{U star}). We conclude that%
\begin{align*}
\mathbb{E}\left[  \left\langle \theta_{\epsilon,N}\left(  t\right)
-\theta_{\epsilon}\left(  t\right)  ,\phi\right\rangle ^{2}\right]   &
\leq\mathbb{E}\int_{0}^{t}\left\Vert \mathbb{Q}_{N}^{0}\right\Vert
_{L^{2}\rightarrow L^{2}}\left\Vert \theta_{0}\right\Vert _{L^{\infty}}%
^{2}\left\Vert \phi\right\Vert _{L^{2}}^{2}ds\\
&  =T\left\Vert \mathbb{Q}_{N}^{0}\right\Vert _{L^{2}\rightarrow L^{2}%
}\left\Vert \theta_{0}\right\Vert _{L^{\infty}}^{2}\left\Vert \phi\right\Vert
_{L^{2}}^{2},
\end{align*}
for $t\in\left[  0,T\right]  $.
\end{proof}

\subsection{Link with the mean field equation, up to a
commutator\label{Subsection commutator}}

Consider now equation (\ref{eq 1}) without the assumption that the vector
fields $\sigma_{k}$ are constant; assume them smooth, bounded and divergence
free. Assume $\theta_{0}\in L^{2}\left(  \mathbb{R}^{2}\right)  $. As outlined
in the previous subsection, introducing the operator $A$ as above but with
$\kappa_{T}=0$ (assuming therefore $\kappa>0$) and the associated analytic
semigroup $e^{tA}$, $t\geq0$, one can study pathwise the equation in mild form%
\begin{equation}
\theta_{\epsilon}\left(  t\right)  =e^{tA}\theta_{0}+\sum_{k\in K}\int_{0}%
^{t}e^{\left(  t-s\right)  A}\left(  \sigma_{k}\cdot\nabla\right)
\theta_{\epsilon}\left(  s\right)  \frac{d\mathcal{G}_{s}^{k,\epsilon}}{ds}ds
\label{mild last section}%
\end{equation}
and prove that there exists a unique solution of class%
\[
\theta_{\epsilon}\in C\left(  \left[  0,T\right]  ;L^{2}\left(  \mathbb{R}%
^{2}\right)  \right)  \cap L^{2}\left(  0,T;W^{1,2}\left(  \mathbb{R}%
^{2}\right)  \right)  .
\]
Moreover, it is measurable in the random parameter. Moreover, it holds
\[
\sup_{t\in\left[  0,T\right]  }\left\Vert \theta_{\epsilon}\left(  t\right)
\right\Vert _{H}^{2}+\kappa\int_{0}^{T}\left\Vert \theta_{\epsilon}\left(
s\right)  \right\Vert _{V}^{2}ds\leq\left\Vert \theta_{0}\right\Vert _{H}%
^{2}.
\]
Solving the equation from a generic initial time $s\geq0$ as indicated in the
previous subsection, the solution defines a family $U_{\epsilon}\left(
t,s,\omega\right)  $ of bounded linear operators on $L^{2}\left(
\mathbb{R}^{2}\right)  $, for $t\geq s\geq0$, satisfying%
\[
U_{\epsilon}\left(  t,s,\omega\right)  U_{\epsilon}\left(  s,0,\omega\right)
=U_{\epsilon}\left(  t,0,\omega\right)
\]%
\[
U_{\epsilon}\left(  s,s,\omega\right)  =Id
\]%
\[
\theta_{\epsilon}\left(  t\right)  =U_{\epsilon}\left(  t,0\right)  \theta
_{0}.
\]
Precisely, $U_{\epsilon}\left(  t,s\right)  \psi$ satisfies%
\[
U_{\epsilon}\left(  t,s\right)  \psi=e^{tA}\psi+\sum_{k\in K}\int_{s}%
^{t}e^{\left(  t-r\right)  A}\left(  \sigma_{k}\cdot\nabla\right)
U_{\epsilon}\left(  r,s\right)  \psi\frac{d\mathcal{G}_{r}^{k,\epsilon}}%
{dr}dr.
\]
In this case we are not able to close the equation for the expected value
$\mathbb{E}\left[  \theta_{\epsilon}\left(  t\right)  \right]  $ and our aim
therefore is only to estimate its distance from the solution of the mean field
equation (\ref{eq 2}).

In order to see the difficulty, let us consider equation
(\ref{mild last section}) in weak form on a test function $\phi\in
C_{c}^{\infty}\left(  \mathbb{R}^{2}\right)  $
\[
\left\langle \theta_{\epsilon}\left(  t\right)  ,\phi\right\rangle
=\left\langle e^{tA}\theta_{0},\phi\right\rangle -\sum_{k\in K}\int_{0}%
^{t}\left\langle \theta_{\epsilon}\left(  s\right)  ,\left(  \sigma_{k}%
\cdot\nabla\right)  e^{\left(  t-s\right)  A}\phi\right\rangle \frac
{d\mathcal{G}_{s}^{k,\epsilon}}{ds}ds.
\]
Then, similarly to the strategy described in Section \ref{Section Commutative}%
, by means of formula (\ref{basic formula}) we rewrite the stochastic integral
a a Skorohod integral (mean zero) plus a trace%
\begin{align*}
\left\langle \theta_{\epsilon}\left(  t\right)  ,\phi\right\rangle  &
=\left\langle e^{tA}\theta_{0},\phi\right\rangle +M_{t}\\
&  -\sum_{k\in K}\int_{0}^{t}\left\langle \left\langle D_{\cdot}^{\left(
k\right)  }\theta_{\epsilon}\left(  s\right)  ,\frac{1}{2\epsilon}1_{\left[
(s-\epsilon)_{+},s+\epsilon\right]  }\right\rangle _{\mathcal{H}},\left(
\sigma_{k}\cdot\nabla\right)  e^{\left(  t-s\right)  A}\phi\right\rangle ds,
\end{align*}
where $M_{t}$ has zero mean. Therefore%
\begin{align*}
\left\langle \mathbb{E}\left[  \theta_{\epsilon}\left(  t\right)  \right]
,\phi\right\rangle  &  =\left\langle e^{tA}\theta_{0},\phi\right\rangle \\
&  -\sum_{k\in K}\int_{0}^{t}\left\langle \left\langle \mathbb{E}\left[
D_{\cdot}^{\left(  k\right)  }\theta_{\epsilon}\left(  s\right)  \right]
,\frac{1}{2\epsilon}1_{\left[  (s-\epsilon)_{+},s+\epsilon\right]
}\right\rangle _{\mathcal{H}},\left(  \sigma_{k}\cdot\nabla\right)  e^{\left(
t-s\right)  A}\phi\right\rangle ds.
\end{align*}
In the commutative case $D_{\cdot}^{\left(  k\right)  }\theta_{\epsilon
}\left(  s\right)  $ can be expressed by means of $\theta_{\epsilon}\left(
s\right)  $ and we find a closed equation for $\mathbb{E}\left[
\theta_{\epsilon}\left(  t\right)  \right]  $. Indeed, by Lemma
\ref{lemma link Malliavin derivative to function}, we have%
\begin{align*}
D_{r}^{\left(  k\right)  }\theta_{\epsilon}\left(  t\right)   &
=\chi_{\epsilon}\left(  t,r\right)  \left(  \sigma_{k}\cdot\nabla\right)
\theta_{\epsilon}\left(  t\right) \\
\chi_{\epsilon}\left(  t,r\right)   &  =\frac{1}{2\epsilon}\int_{0}%
^{t}1_{\left[  (s-\epsilon)_{+},s+\epsilon\right]  }\left(  r\right)  ds.
\end{align*}
Now, without commutation, we have only the following result.

\begin{lemma}%
\[
D_{r}^{\left(  k\right)  }\theta_{\epsilon}\left(  t\right)  =\chi_{\epsilon
}\left(  t,r\right)  U_{\epsilon}\left(  t,r\right)  \left(  \sigma_{k}%
\cdot\nabla\right)  \theta_{\epsilon}\left(  r\right)  .
\]

\end{lemma}

\begin{proof}
Indeed, from (\ref{mild last section}) we get%
\begin{align*}
D_{r}^{\left(  k\right)  }\theta_{\epsilon}\left(  t\right)   &
=\sum_{k^{\prime}\in K}\int_{r}^{t}e^{\left(  t-s\right)  A}\left(
\sigma_{k^{\prime}}\cdot\nabla\right)  D_{r}^{\left(  k\right)  }%
\theta_{\epsilon}\left(  s\right)  \frac{d\mathcal{G}_{s}^{k^{\prime}%
,\epsilon}}{ds}ds\\
&  +\int_{0}^{t}e^{\left(  t-s\right)  A}\left(  \sigma_{k}\cdot\nabla\right)
\theta_{\epsilon}\left(  s\right)  \frac{dD_{r}^{\left(  k\right)
}\mathcal{G}_{s}^{k,\epsilon}}{ds}ds.
\end{align*}
Then we use (\ref{approximate derivative}) to express $D_{r}^{\left(
k\right)  }\mathcal{G}_{s}^{k,\epsilon}$ and get%
\begin{align*}
D_{r}^{\left(  k\right)  }\theta_{\epsilon}\left(  t\right)   &
=\sum_{k^{\prime}\in K}\int_{r}^{t}e^{\left(  t-s\right)  A}\left(
\sigma_{k^{\prime}}\cdot\nabla\right)  D_{r}^{\left(  k\right)  }%
\theta_{\epsilon}\left(  s\right)  \frac{d\mathcal{G}_{s}^{k^{\prime}%
,\epsilon}}{ds}ds\\
&  +e^{\left(  t-r\right)  A}\left(  \sigma_{k}\cdot\nabla\right)
\theta_{\epsilon}\left(  r\right)  \chi_{\epsilon}\left(  t,r\right)
\end{align*}
which leads to the result by uniqueness for the equation defining
$U_{\epsilon}\left(  t,r\right)  $.
\end{proof}

\medskip
The problem is that we cannot commute $U_{\epsilon}\left(  t,r\right)  \left(
\sigma_{k}\cdot\nabla\right)  $ with $\left(  \sigma_{k}\cdot\nabla\right)
U_{\epsilon}\left(  t,r\right)  $, otherwise we would have%
\begin{align*}
D_{r}^{\left(  k\right)  }\theta_{\epsilon}\left(  t\right)   &
=\chi_{\epsilon}\left(  t,r\right)  \left(  \sigma_{k}\cdot\nabla\right)
U_{\epsilon}\left(  t,r\right)  \theta_{\epsilon}\left(  r\right) \\
&  =\chi_{\epsilon}\left(  t,r\right)  \left(  \sigma_{k}\cdot\nabla\right)
\theta_{\epsilon}\left(  t\right),
\end{align*}
(due to $\theta_{\epsilon}\left(  r\right)  =U_{\epsilon}\left(  r,0\right)
\theta_{0}$ and $U_{\epsilon}\left(  t,r\right)  U_{\epsilon}\left(
r,0\right)  =U_{\epsilon}\left(  t,0\right)  $) like in the commuting case.
Summarizing, until now we have established the following.

\begin{lemma}%
\[
\left\langle \mathbb{E}\left[  \theta_{\epsilon}\left(  t\right)  \right]
,\phi\right\rangle =\left\langle e^{tA}\theta_{0},\phi\right\rangle
\]%
\[
-\sum_{k\in K}\int_{0}^{t}\left\langle \left\langle \chi_{\epsilon}\left(
s,\cdot\right)  \mathbb{E}\left[  U_{\epsilon}\left(  s,\cdot\right)  \left(
\sigma_{k}\cdot\nabla\right)  \theta_{\epsilon}\left(  \cdot\right)  \right]
,\frac{1}{2\epsilon}1_{\left[  (s-\epsilon)_{+},s+\epsilon\right]
}\right\rangle _{\mathcal{H}},\left(  \sigma_{k}\cdot\nabla\right)  e^{\left(
t-s\right)  A}\phi\right\rangle ds.
\]

\end{lemma}
Adding and subtracting the term with $\left(  \sigma_{k}\cdot\nabla\right)
U_{\epsilon}\left(  t,r\right)  $ in place of $U_{\epsilon}\left(  t,r\right)
\left(  \sigma_{k}\cdot\nabla\right)  $ we have%
\[
\left\langle \mathbb{E}\left[  \theta_{\epsilon}\left(  t\right)  \right]
,\phi\right\rangle =\left\langle e^{tA}\theta_{0},\phi\right\rangle
\]%
\begin{align*}
&  -\sum_{k\in K}\int_{0}^{t}\left\langle \left\langle \chi_{\epsilon}\left(
s,\cdot\right)  ,\frac{1}{2\epsilon}1_{\left[  (s-\epsilon)_{+},s+\epsilon
\right]  }\right\rangle _{\mathcal{H}}\left(  \sigma_{k}\cdot\nabla\right)
\mathbb{E}\left[  \theta_{\epsilon}\left(  s\right)  \right]  ,\left(
\sigma_{k}\cdot\nabla\right)  e^{\left(  t-s\right)  A}\phi\right\rangle ds\\
&  +\sum_{k\in K}\int_{0}^{t}\left\langle \left\langle \chi_{\epsilon}\left(
s,\cdot\right)  R_{\epsilon,k}\left(  s,\cdot\right)  ,\frac{1}{2\epsilon
}1_{\left[  (s-\epsilon)_{+},s+\epsilon\right]  }\right\rangle _{\mathcal{H}%
},\left(  \sigma_{k}\cdot\nabla\right)  e^{\left(  t-s\right)  A}%
\phi\right\rangle ds,
\end{align*}
where for shortness of notations we have set%
\begin{equation}
R_{\epsilon,k}\left(  s,r\right)  =\mathbb{E}\left[  \left(  U_{\epsilon
}\left(  s,r\right)  \left(  \sigma_{k}\cdot\nabla\right)  -\left(  \sigma
_{k}\cdot\nabla\right)  U_{\epsilon}\left(  s,r\right)  \right)
\theta_{\epsilon}\left(  r\right)  \right] . \label{defin commut}%
\end{equation}
With the notations of Section \ref{Section Commutative}, we have%
\[
\left\langle \mathbb{E}\left[  \theta_{\epsilon}\left(  t\right)  \right]
,\phi\right\rangle =\left\langle e^{tA}\theta_{0}+\int_{0}^{t}e^{\left(
t-s\right)  A}\mathcal{L}\mathbb{E}\left[  \theta_{\epsilon}\left(  s\right)
\right]  d\mathcal{V}_{\epsilon}\left(  s\right)  ,\phi\right\rangle
\]%
\[
+\sum_{k\in K}\int_{0}^{t}\left\langle \left\langle \chi_{\epsilon}\left(
s,\cdot\right)  R_{\epsilon,k}\left(  s,\cdot\right)  ,\frac{1}{2\epsilon
}1_{\left[  (s-\epsilon)_{+},s+\epsilon\right]  }\right\rangle _{\mathcal{H}%
},\left(  \sigma_{k}\cdot\nabla\right)  e^{\left(  t-s\right)  A}%
\phi\right\rangle ds
\]
or%
\begin{align*}
\mathbb{E}\left[  \theta_{\epsilon}\left(  t\right)  \right]   &
=e^{tA}\theta_{0}+\int_{0}^{t}e^{\left(  t-s\right)  A}\mathcal{L}%
\mathbb{E}\left[  \theta_{\epsilon}\left(  s\right)  \right]  d\mathcal{V}%
_{\epsilon}\left(  s\right) \\
&  -\sum_{k\in K}\int_{0}^{t}e^{\left(  t-s\right)  A}\left(  \sigma_{k}%
\cdot\nabla\right)  \left\langle \chi_{\epsilon}\left(  s,\cdot\right)
R_{\epsilon}\left(  s,\cdot\right)  ,\frac{1}{2\epsilon}1_{\left[
(s-\epsilon)_{+},s+\epsilon\right]  }\right\rangle _{\mathcal{H}}ds.
\end{align*}
Notice that the mean field equation is
\[
\overline{\theta}\left(  t\right)  =e^{tA}\theta_{0}+\int_{0}^{t}e^{\left(
t-s\right)  A}\mathcal{L}\overline{\theta}\left(  s\right)  d\gamma\left(
s\right)  .
\]
The closedness of $\mathbb{E}\left[  \theta_{\epsilon}\left(  t\right)
\right]  $ to $\overline{\theta}\left(  t\right)  $ depends on the smallness
of the average commutator $R_{\epsilon,k}\left(  s,r\right)  $. Estimates on
$R_{\epsilon,k}\left(  s,r\right)  $ seem possible but those we have found
until now do not deserve to be reported, so we postpone this subject to future research.

\bigskip
\textbf{ACKNOWLEDGEMENTS}. This research was influenced by discussions with
several experts and in particular, for the first author, by James Michael
Leahy and Umberto Pappalettera. The research of the first author is funded by
the European Union (ERC, NoisyFluid, No. 101053472). Views and opinions
expressed are however those of the authors only and do not necessarily reflect
those of the European Union or the European Research Council. Neither the
European Union nor the granting authority can be held responsible for them.
The research of the second named author was partially supported by the
ANR-22-CE40-0015-01 (SDAIM).

\bibliographystyle{plain}
\bibliography{../../BIBLIO_FILE/BiblioFrancoFrancesco}
 \end{document}